\documentclass[conference]{IEEEtran}
\IEEEoverridecommandlockouts
\usepackage[T1]{fontenc}
\usepackage[latin9]{inputenc}
\usepackage[compress]{natbib}
\usepackage{amsthm}
\usepackage{amsmath}
\usepackage{algcompatible}
\usepackage{algorithm}
\usepackage{algorithmicx}
\usepackage{algpseudocode}    
\usepackage{booktabs}
\usepackage{comment}
\usepackage{enumerate}
\usepackage{graphicx}
\usepackage{amssymb}
\usepackage{latexsym}
\usepackage{epstopdf}
\usepackage{color}
\usepackage{bbm}
\usepackage{needspace}
\usepackage{colortbl}
\usepackage[english]{babel}
\usepackage{tikz}
\usepackage{caption}
\usepackage{subcaption}
\usetikzlibrary{arrows,automata}
\usetikzlibrary{positioning}
\usepackage{filecontents}
\setcitestyle{numbers}
\DeclareCaptionFont{mysize}{\fontsize{8}{9.6}\selectfont}
\ifCLASSINFOpdf
\else
\fi
\setcitestyle{square}
\theoremstyle{plain}
\newtheorem{thm}{Theorem}
  \theoremstyle{plain}
  \newtheorem{lem}{Lemma}

\newtheorem{remark}{Remark}

\newtheorem{define}{Definition}

\newtheorem{prob}{Problem}

\newtheorem{assumption}{Assumption}
\algnewcommand{\LeftComment}[1]{\Statex \(\triangleright\) #1}
\makeatother
\newcommand{\multiline}[1]{%
	\begin{tabularx}{\dimexpr\linewidth-\ALG@thistlm}[t]{@{}X@{}}
		#1
	\end{tabularx}
}
\makeatother
\usepackage{babel}
\begin{document}
	\addtolength{\voffset}{0.1in}
\title{Resilient Structural Stabilizability of Undirected Networks}
\author{Jingqi Li, Ximing Chen, S{\'e}rgio Pequito, George J. Pappas, Victor M. Preciado%, Ximing Chen, George J. Pappas, and Victor M. Preciado%
	\thanks{This research was supported in part by the US National Science Foundation under the grant CAREER-ECCS-1651433.
		
		Jingqi Li, Ximing Chen, George J. Pappas, and Victor M. Preciado are with the Department of Electrical and Systems Engineering	at the University of Pennsylvania, Philadelphia, PA 19104.  e-mail: \{jingqili, ximingch, pappasg, preciado\}@seas.upenn.edu. S{\'e}rgio Pequito is with the Department of Industrial and Systems Engineering at the Rensselaer Polytechnic Institute, Troy, NY 12180-3590. email: goncas@rpi.edu.
}}
\maketitle
\begin{abstract}
In this paper, we consider the \emph{structural stabilizability} problem of undirected networks. More specifically, we are tasked to infer the stabilizability of an undirected network from its underlying topology, where the undirected networks are modeled as continuous-time linear time-invariant (LTI) systems involving symmetric state matrices. Firstly, we derive a graph-theoretic necessary and sufficient condition for structural stabilizability of undirected networks. Then, we propose a method to infer the maximum dimension of stabilizable subspace solely based on the network structure. Based on these results, on one hand, we study the optimal actuator-disabling attack problem, i.e., removing a limited number of actuators to minimize the maximum dimension of stabilizable subspace. We show this problem is NP-hard. On the other hand, we study the optimal recovery problem with respect to the same kind of attacks, i.e., adding a limited number of new actuators such that the maximum dimension of stabilizable subspace is maximized. We prove the optimal recovery problem is also NP-hard, and we develop a $(1-1/e)$ approximation algorithm to this problem.
\end{abstract}

\section{Introduction}
In recent years, the control of networked dynamical systems has attracted a great amount of research interest \cite{ren2005survey,simpson2017voltage,oh2015survey}. It is of particular interest to study the asymptotic stabilizability of network control systems, i.e., the ability ensuring that all the system states can be steered to the origin by injecting proper controls, such as the undirected consensus network \cite{ren2005survey}, voltage stabilization of grids \cite{simpson2017voltage}, and formation control with undirected communication links \cite{oh2015survey}.

The existing results on stabilizability analysis highly rely on the assumption that the system parameters can be exactly acquired, which is often violated in practice,~(-~see~\cite{syrmos1994static,lin2009stability,chapman2015state} and the references therein). It has been shown that the topological structure of a network, which can be obtained accurately, can be exploited to infer the required conditions to ensure the controllability of a network system efficiently \cite{lin1974structural,dion2003generic,liu2011controllability}. This motivates us to investigate the interplay between the network's structure and the stabilizability of a network. 

Assessing the stabilizability from the structural information on the system dynamics model has been an active topic of research \cite{kirkoryan2014decentralized,pichai1984graph,pajic2013topological,pequito2016framework}. {\color{black}However, in \cite{kirkoryan2014decentralized}, the authors assumed no control input and proposed conditions on the sparsity pattern of symmetric state matrices such that a specific sparsity pattern sustains a Hurwitz stable state matrix. }%However, the approach in \cite{kirkoryan2014decentralized} assumes no symmetry in the system dynamics and, therefore, they cannot be applied to the undirected networked system due to the symmetries of the system parameters therein. 
%However, the approach in \cite{kirkoryan2014decentralized} assumes no control inputs and, therefore, the
In addition, the problem considered in \cite{pichai1984graph,pajic2013topological,pequito2016framework} is the arbitrary pole placement through output feedback, which is sufficient but not necessary for the stabilizability.

Stabilizability is a crucial concept in network security \cite{pang2012design} and there has been a tremendous effort invested into the control of networks under malicious attacks \cite{pang2012design,pequito2017robust,liu2013minimum,pequito2015analysis,moothedath2017verifying,zhang2017edge,chen2018minimal,zhu2014performance,de2015input,rai2012vulnerable}. The problems of adding extra actuators/sensors to ensure controllability/observability under attacks are addressed in \cite{pequito2017robust,liu2013minimum}. The problem of maintaining stabilization under the uncertain feedback-channel failure is considered in \cite{pequito2015analysis,moothedath2017verifying}. In \cite{zhang2017edge,chen2018minimal}, the problem of optimal attack/recovery on structural controllability is investigated. Although the problems of stabilization under various attacks such as deception attack \cite{pang2012design}, replay attacks \cite{zhu2014performance}, denial-of-service \cite{de2015input} and destabilizing attacks \cite{rai2012vulnerable}, have been widely studied, the crucial problem of optimal attack against stabilizability by manipulating network topological structure, e.g., removing or adding actuators, has not been fully investigated. Moreover, to the best of the authors' knowledge, our paper considers for the first time the problems of optimal attack and recovery on the stabilizable subspace of a network, i.e., the number of stabilizable states or nodes in a network.%, which is also a difference from the most existing literature.%When the system is not controllable under attacks, the \emph{controllable subspace} \cite{hosoe1980determination} is introduced to describe the number of system states which are controllable. %Similarly, we can introduce the \emph{stabilizable subspace} to characterize the number of system states that can be stabilized. 
%Although attacks on damaging controllability has been widely studied, the problem of optimal attack on minimizing controllable subspace by disabling actuators has still not been fully investigated. 

Specifically, in this paper, we consider the structural stabilizability problem, and the contributions of this paper are four-fold. First, we derive a graph-theoretic necessary and sufficient condition for structural stabilizability of undirected networks. Second, we propose graph-theoretic methods to infer the generic dimension of controllable subspace and the maximum stabilizable subspace of an undirected network system. Third, we formulate the optimal actuator-disabling attack problem, where the attacker disables a limited number of actuators such that the maximum stabilizable subspace is minimized. We prove this problem is NP-hard. Finally, we formulate the optimal recovery problem, where a defender activates a limited number of new actuators such that the dimension of the stabilizable subspace is maximized. We prove this problem is NP-hard, and we propose a $(1-1/e)$ approximation algorithm.

The rest of the paper is organized as follows. In Section~\ref{sec:Problems}, we formulate the problems considered in this paper. In Section~\ref{sec:Preliminaries}, we recall several crucial preliminaries. We present the main results in Sections~\ref{sec:Results} and \ref{sec:Design} -- the proofs are relegated to the Appendix. In Section~\ref{sec:Examples}, we present examples to illustrate our results. Finally, Section~\ref{sec:conclusion} concludes this paper.% All the proofs are included in the Appendix.
%{\color{blue}To be amplified...
%
%1. Infer stabilizability from topology (Belabbas) + previous results on structural stabilizability 
%
%2. Symmetric system results, RC and RL networks, single energy system
%
%3. network security, structural approach
%}
\section{Problem Formulations}\label{sec:Problems}
We consider networks whose interconnection between states are captured by a symmetric linear time-invariant (LTI) system,  described by
\begin{equation}\label{sys}\small
\dot{x}=Ax+Bu,
\end{equation}
where $x\in\mathbb{R}^n$ and $u\in\mathbb{R}^m$ are state vector and input vector, respectively. We refer to matrices $A=A^\top\in\mathbb{R}^{n\times n}$ and $B\in\mathbb{R}^{n\times m}$ as the state matrix and input matrix, respectively. Hereafter, we use the pair $(A,B)$ to represent the system~\eqref{sys}. %It will be of interest to answer the question, given a structure of such a system, whether it is possible to make it stable by tuning its independent free parameters. So, the notion of \emph{structural stabilizability} is introduced.

%The interconnection between states can often be captured by a network in common problems concerning about, for example, RC networks \cite{anderson2013network}, multiagent networks \cite{ren2005survey}, and brain networks \cite{gu2015controllability}. In these applications, the state matrix is assumed to be symmetric. In this paper, we impose the following assumption on the state matrix.

%\begin{assumption}\label{ap}
%	The state matrix $A\in\mathbb{R}^{n\times n}$ is symmetric matrix, i.e., $A=A^{\top}$. 
%\end{assumption}

%Although it is generally difficult to obtain the exact system parameters, interconnections between states can be observed easily in most of the practical systems~\cite{pequito2015analysis}. 
In order to infer the properties of a system modeled by \eqref{sys} from its structure, we introduce some necessary concepts on structured matrices.
\begin{define}[Structured and Symmetrically Structured Matrices]
	A matrix $\bar{M}\in\{0,\star\}^{n\times m}$ is called a \emph{structured matrix}, if $[\bar{M}]_{ij}$, the $(i,j)$-th entry of $\bar{M}$, is either a fixed zero or an independent free parameter, denoted by $\star.$ In particular, a matrix $\bar{M}\in\{0,\star\}^{n\times n}$ is \emph{symmetrically structured}, if the value of the free parameter associated with $[\bar{M}]_{ji}$ is constrained to be the same as the value of the free parameter associated with $[\bar{M}]_{ij}$, for all $i$ and $j$. % $[\bar{M}]_{ij}$ is either a fixed zero or $\star$ for all $j\ge i$, and  $[\bar{M}]_{ji}$ and $[\bar{M}]_{ij}$ are parameterized by a same $\star$-variable, for all $j < i$.
\end{define}
We refer to $\tilde{M}$ as a \emph{numerical realization} of a (symmetrically) structured matrix $\bar{M}$ if $\tilde{M}$ is a matrix obtained by assigning real numbers to $\star$-parameters in~$\bar{M}$.

Given a pair $(A,B)$, we let the pair $(\bar{A},\bar{B})$ denote the \emph{structural pattern} of the system $(A, B),$ where $\bar{A}\in\{0,\star\}^{n\times n}$ is a symmetrically structured matrix such that $[\bar{A}]_{ij}=\star$ if $[A]_{ij}\ne 0$ and $[\bar{A}]_{ij}=0$ otherwise. The structured matrix $\bar{B}\in\{0,\star\}^{n\times m}$ is defined similarly. %By Assumption~\ref{ap}, we consider $\bar{A}$ to be symmetrically structured, while the $\bar{B}$ is a structured matrix. 

Recall that a system is stabilizable if and only if the uncontrollable eigenvalues are asymptotically stable \cite[Section 2.4]{wood1986linear}. Hence, to study stabilizability, it is necessary to first investigate controllability. 
Next, we recall the notion of \emph{structural controllability}.%, which facilitates our~later~work~on~inferring~controllability~from~the~structural~pattern~of~a system.% can be considered as an interpretation of controllability under the context of structural system theory.
\begin{define}[Structural Controllability~\cite{lin1974structural}]
A structural pair $(\bar{A},\bar{B})$ is structurally controllable if there exists a numerical realization $(\tilde{A},\tilde{B})$ such that the controllability matrix $Q(\tilde{A},\tilde{B}):=[\tilde{B},\tilde{A}\tilde{B},\cdots,\tilde{A}^{n-1}\tilde{B}]$ has full row rank.
\end{define}
%We know that the dimension of controllable subspace of almost all realization of a structural pair $(\bar{A},\bar{B})$ equals to a fixed number, say $n_c.$ 
%Recall that a system is stabilizable if and only if the uncontrollable eigenvalues are stable.
Similarly, we define \emph{structural stabilizability} as follows:
\begin{define}[Structural Stabilizability]
A structural pair $(\bar{A},\bar{B})$ is said to be structurally stabilizable if there exists a stabilizable numerical realization $(\tilde{A},\tilde{B})$.
\end{define}

\begin{remark}
Stabilizability is not a generic property \cite{dion2003generic}, yet the structural stabilizability of $(\bar{A},\bar{B})$ implies the existence of a numerical realization $(\tilde{A},\tilde{B})$ such that $(\tilde{A},\tilde{B})$ is stabilizable. It is a necessary condition for the stabilizability of any realization $(\tilde{A},\tilde{B})$ of a structural pair $(\bar{A},\bar{B})$. 
\end{remark}
%\begin{remark}
%	Although the controllability is a generic property, the stabilizability is not, in the sense that the property holds for almost all numerical realizations $(\tilde{A},\tilde{B})$~\cite{dion2003generic}.
%\end{remark}

%To show the that stabilizability is a generic property, we can consider that a system is unstabilizable implies that it is uncontrollable. On the other hand, $(\bar{A},\bar{B})$ is structurally controllable if and only if all controllable pairs structurally equivalent to $(\bar{A},\bar{B})$ lies in a proper variety -- Proposition 2 by Shields and Pearson. Thus, structural stabilizability is also a generic property. 
In the next two subsections, we will be focusing on two different main threads: (\emph{i}) analysis, and (\emph{ii}) design.
\subsection{Analysis of Structural Stabilizability}
In this subsection, we first formulate the problem of characterizing structural stabilizability using only the structural pattern of a pair, as stated below:
%{\color{black} We can also consider the system being strong structurally stabilizable. A structural pair is called structurally strongly stabilizable whatever values (other than zero) the independent free parameters of the system may take, the system is stabilizable. } 

%{\color{black} Notice that for a continuous-time system, a system is stabilizable if and only if $\textrm{rank}([\lambda I - A, B]) = n$ for all $\lambda \in \mathbb{C}$ with $Re(\lambda) > 0.$ Similarly, a discrete-time system is stabilizable if and only if $\textrm{rank}([\lambda I - A, B]) = n$ for all unstable eigenvalues of $A.$}

\begin{prob}\label{p0}
Given a continuous-time linear time-invariant pair $(A,B)$, we denote by $(\bar{A},\bar{B})$ the structural pattern of $(A,B),$ where $\bar{A}\in\{0,\star\}^{n\times n}$ is symmetrically structured. Find a necessary and sufficient condition such that $(\bar{A},\bar{B})$ is structurally stabilizable.
\end{prob}
%{\color{red} I do not understand the following lemma. It is too early to tell the readers this. Either explain what is the difference or state this lemma after you present your results.
%\begin{remark}
%Notice that the stabilizability does not lead to the arbitary pole placement of the system \eqref{sys}, which differentiates this paper from the existing results on stabilization of LTI systems from the structural perspective \cite{pichai1984graph,pajic2013topological,pequito2016framework}. 
%\end{remark}}
In addition to the above problem, we also consider how ``unstabilizable'' a system is, when a system is not stabilizable. To characterize the ``unstabilizability'', we propose using the dimension of the stabilizable subspace of a system, which can be stated as follows:
\begin{define}[Stabilizable Subspace \cite{hautus1980b}]\label{dss}
Given a pair $(A,B)$, where $A\in\mathbb{R}^{n\times n}$ and $B\in\mathbb{R}^{n\times m}$, a subspace $S\subseteq\mathbb{R}^n$ is said to be the stabilizable subspace of $(A,B)$ if for $\forall x(0)\in S$, there exists a control input $u(t)\in\mathbb{R}^m,$ for $t\ge 0$, such that
	\begin{equation*}\small
		\lim_{t\to \infty} x(t)=\mathbf{0}.
	\end{equation*}
\end{define}
%The \emph{dimension} of the stabilizable subspace can be interpreted as the total number of stabilizable states in the state vector $x\in\mathbb{R}^n.$ 
As a special case, if a pair $(A,B)$ is stabilizable, then $S=\mathbb{R}^n.$ 
%
%Although stabilizability is not a generic property \cite{dion2003generic}, the structural stabilizability of $(\bar{A},\bar{B})$ implies the existence of a numerical realization $(\tilde{A},\tilde{B})$ such that $(\tilde{A},\tilde{B})$ is stabilizable, i.e., it is a necessary condition for the stabilizability of any realization $(\tilde{A},\tilde{B})$ of a structural pair $(\bar{A},\bar{B})$. 
Moreover, we aim to determine the maximum dimension of stabilizable subspace, denoted by $\textrm{m-dim}(\bar{A},\bar{B}),$ among all numerical realizations of $(\bar{A},\bar{B}).$ Formally, we can state this problem as follows.
\begin{prob}\label{p1}
Given a structural pair $(\bar{A},\bar{B})$, where $\bar{A}$ is symmetrically structured, find $\emph{m-dim}(\bar{A},\bar{B}).$
%the maximum dimension of stabilizable subspace among all the numerical realizations $(\tilde{A},\tilde{B})$ of $(\bar{A},\bar{B}).$
\end{prob}

Upon these problems that concern mainly with the analysis of structural stabilizability, we can now focus on the design aspect of these problems in the following subsection. 
%We denote by $\mbox{\upshape{m-dim}}(\bar{A},\bar{B}),$ the solution to the above problem.
%As noted above, the structural stabilizability is a necessary condition for the stabilizbaility of a pair $(\tilde{A},\tilde{B})$. %As noted in \cite{pajic2013stabilizability}, stabilizability is a crucial property in network security problems. For example, if an attacker destroys the structural stabilizability of the pair $(\bar{A},\bar{B})$, then any numerical realization $(\tilde{A},\tilde{B})$ will not be stabilizable. Consequently, in the view of Definition~\ref{dss}, the states of the system cannot be steered to $0.$ 
\subsection{Optimal Actuator-Attack and Recovery Problems}
Stabilizability plays a key role in network security \cite{pang2012design}. In this paper, we also consider the network resilient problems. More specifically, we assume that an attacker aims to minimize the maximum dimension of the stabilizable subspace by removing a certain amount of actuation capabilities, i.e., inputs. We formalize this problem as follows.
\begin{prob}[Optimal Actuator-disabling Attack Problem]\label{p2}
Consider a stuctural pair $(\bar{A},\bar{B})$, where $\bar{A}\in\{0,\star\}^{n\times n}$ is symmetrically structured, and $\bar{B}\in\{0,\star\}^{n\times m}$ is a structured matrix. Let the set $\Omega$ be $\Omega=[m]$, where $[m]:=\{1,2,\cdots,m\}$. %Let $\mathcal{U}=\{u_i\}_{i=1}^m$ be the set of all the inputs to the system. 
Given a budget $k\in\mathbb{N}$, find
\begin{equation}\label{eq:attack}\small
\begin{split}
	\mathcal{J}^*=\arg & \min_{\mathcal{J}\subseteq \Omega} \emph{m--dim}(\bar{A},\bar{B}(\Omega\setminus\mathcal{J}))\\
	\emph{s.t.}&\  |\mathcal{J}|\le k,
	%				\mathcal{U}_r^*=\arg & \min_{\mathcal{U}_r\subseteq \mathcal{U}} \emph{g--dim}(\mathcal{X\cup (U\setminus U}_r),\mathcal{E_{X,X}}\cup \mathcal{E}_{(\mathcal{U\setminus U}_r),\mathcal{X}})\\
	%				s.t.&\  |\mathcal{U}_r|\le c,
	\end{split}
	\end{equation}
	where $\bar{B}(\mathcal{I})\in\{0,\star\}^{n\times |\mathcal{I}|}$ is a matrix formed by the columns of $\bar{B}$ indexed by $\mathcal{I}$, for some $\mathcal{I}\subseteq \Omega$.
\end{prob}
In other words, Problem~\ref{p2} concerns about finding an optimal strategy to attack the stabilizability of a network using a fixed budget. Meanwhile, it is also of interest to consider the perspective of a system's designer (or, defender) that is concerned with the resilience of the network, i.e., how to maximize the dimension of stabilizable subspace by adding actuation capabilities (i.e., inputs) to the system:
\begin{prob}[Optimal Recovery Problem]\label{p3}
Consider a structural pair $(\bar{A},\bar{B})$, where $\bar{A}\in\{0,\star\}^{n\times n}$ is symmetrically structured and $\bar{B}\in\{0,\star\}^{n\times m}$ is structured. Let $\mathcal{U}_{can}$, where $|\mathcal{U}_{can}|=m'$, be the set of candidate inputs that can be added to the system, and let $\bar{B}_{\mathcal{U}_{can}}\in\{0,\star\}^{n\times m'}$ be the structured matrix characterizing the interconnection between new inputs and the states in the system. Given a budget $k\in\mathbb{N}$, find
	\begin{equation}\small
	\begin{split}
	\mathcal{J}^*=\arg & \max_{\mathcal{J}\subseteq[m']}\emph{m-dim}(\bar{A},[\bar{B},\bar{B}_{\mathcal{U}_{can}}(\mathcal{J})])\\
	\emph{s.t.}&\ |\mathcal{J}|\le k,
	\end{split}
	\end{equation}
	where $\bar{B}_{\mathcal{U}_{can}}(\mathcal{J})\in\{0,\star\}^{n\times |\mathcal{J}|}$ is a structured matrix formed by the columns in $\bar{B}_{\mathcal{U}_{can}}$ indexed by $\mathcal{J}$, and $[\bar{B},\bar{B}_{\mathcal{U}_{can}}(\mathcal{J})]$ is the concatenation of $\bar{B}$ and $\bar{B}_{\mathcal{U}_{can}}(\mathcal{J})$.
\end{prob}

By the duality between stabilizability and detectability \cite{wood1986linear}, all the results obtained on stabilizability in this paper can be readily used to characterize detectability.% In addition, structural stabilizability is a more strict concept than structural controllability; hence the solutions to Problem \ref{p2} and Problem \ref{p3}, can be applied to concerning structural controllability.
\section{Preliminaries}\label{sec:Preliminaries}
To present solutions to Problems~1 -- 4, we introduce some relevant notions in structural system theory and graph theory.\vspace{-0.1cm}
\subsection{Structural System Theory}
%Consider a (symmetrically) structured matrix $\bar{M}$, we let $n_{\bar{M}}$ be the number of independent free parameters in $\bar{M}$. We associate parameter space $\mathbb{R}^{n_{\bar{M}}}$ with $\bar{M}$.%, so we have $\tilde{M}\in\mathbb{R}^{n_{\bar{M}}}$. %Denote $\lambda_1\ge\lambda_2\ge\dots\ge\lambda_n$ as the eigenvalues of $\tilde{M}$.
Consider a (symmetrically) structured matrix $\bar{M}$. Let $n_{\bar{M}}$ be the number of its independent $\star$-parameters and associate with $\bar{M}$ a parameter space $\mathbb{R}^{n_{\bar{M}}}$. %Subsequently, we use a vector 
Let $\mathbf{p}_{\tilde{M}}=(p_1,\dots,p_{n_{\bar{M}}})^\top\in\mathbb{R}^{n_{\bar{M}}}$ to encode the values of the independent $\star$-entries of $\bar{M}$ of a particular numerical realization $\tilde{M}$. In what follows, a set $V\subseteq\mathbb{R}^n$ is called a \emph{variety} if there exist polynomials $\varphi_1,\dots,\varphi_k$, such that $V=\{x\in\mathbb{R}^n\colon \varphi_i(x)=0,\forall i\in[k]\}$, and $V$ is \emph{proper} when~$V\ne \mathbb{R}^n$. We denote by $V^c=\mathbb{R}^n\setminus V$ its complement. %We denote the \emph{complement} of $V\subseteq\mathbb{R}^n$ as $V^c=\mathbb{R}^n\setminus V$.}
%For example, $S=\{p\in\mathbb{R}^3:p^2_1+p^2_2+p^2_3=0\}$ is a proper variety. 
%	
{%A dimension of a variety described as the range of the column space of a matrix can be captured by the rank, and similarly, 
	
The \emph{term rank} \cite{murota2012systems} of a (symmetrically) structured matrix $\bar{M}$, denoted as $\textrm{t--rank}(\bar{M})$, is the largest integer $k$ such that, for some suitably chosen distinct rows $\{i_{\ell}\}_{\ell=1}^k$ and distinct columns $\{j_{\ell}\}_{\ell=1}^k$, all of the entries $\{ [\bar{M}]_{i_\ell j_\ell} \}_{\ell = 1}^k$ are $\star$-entries. Additionally, a (symmetrically) structured matrix $\bar{M}\in\{0,\star\}^{n\times m}$ is said to have \emph{generic rank} $k$, denoted as $\textrm{g--rank}(\bar{M})=k$, if there exists a numerical realization $\tilde{M}$ of $\bar{M}$, such that $\textrm{rank}(\tilde{M})=k$. Note that, if $\textrm{g--rank}(\bar{M})>0$, then the set of parameters describing all possible realizations when $\textrm{rank}(\tilde{M})<\textrm{g--rank}(\bar{M})$ form a proper variety, \cite{hosoe1980determination}.
	
	Given a structural pair $(\bar{A},\bar{B})$, where $\bar{A}\in\{0,\star\}^{n\times n}$ is symmetrically structured, $(\bar{A},\bar{B})$ is said to be \emph{irreducible}, if there does not exist a permutation matrix $P$ such that\vspace{-0.1cm}
	\begin{equation}\label{eq}\small
	P\bar{A}P^\top=\begin{bmatrix}
	&\bar{A}_{11} & \mathbf{0}\\
	&\mathbf{0} & \bar{A}_{22}
	\end{bmatrix},P\bar{B}=\begin{bmatrix}
	\bar{B}_{1}\\\mathbf{0}
	\end{bmatrix},\vspace{-0.1cm}
	\end{equation}
	where $\bar{A}_{11}\in\{0,\star\}^{p\times p}$, and $\bar{B}_1\in\{0,\star\}^{p\times m}$.\vspace{-0.1cm}

	%	\begin{equation*}\small
	%	V=\{\mathbf{p}_{\tilde{M}}\in\mathbb{R}^{n_{\bar{M}}}:\sum_{\begin{subarray}
	
	%		1\le i_1<\dots<i_{k} \le n\\
	%		1\le j_1<\dots<j_k\le m
	%		\end{subarray}}\left(\det\left([\tilde{M}]_{i_1,\dots,i_{k}}^{j_1,\dots,j_{k}}\right)\right)^2=0 \}
	%	\end{equation*}
	%	is proper \cite{hosoe1980determination}.
	%\begin{thm}
	%	The pair $(A,B)$ is structurally controllable if and only if the following two conditions hold:
	%	\begin{enumerate}[(a)]
	%		\item \label{itm:SC1} every state vertex $x\in\mathcal{X}$ in the system digraph $G(\bar{A},\bar{B})=(\mathcal{X}\cup\mathcal{U}, \mathcal{E}_{\mathcal{X},\mathcal{X}}\cup \mathcal{E}_{\mathcal{U},\mathcal{X}})$ is reachable (from some input vertex $u\in \mathcal U$);
	%		\item \label{itm:SC2} any maximum matching $M$ of the system bipartite graph $\mathcal{B}(\bar{A},\bar{B}) = \mathcal{B}(\mathcal{X}^+\cup\mathcal{U}^+, \mathcal{X}^-, \mathcal{E}_{\mathcal{X}^+,\mathcal{X}^-}\cup \mathcal{E}_{\mathcal{U}^+,\mathcal{X}^-})$ has no right-unmatched vertices.\hfill $\diamond$
	%	\end{enumerate}
	%\end{thm}
	
	%Notice that both conditions in Theorem~\ref{Thm:SControllability} can be verified in polynomial time \cite{dion2003generic}. Hence, one could naively try to ensure both conditions by adding edges iteratively, but such an approach is, in general, non-optimal and does not provide optimality guarantees.
	\subsection{Graph Theory}
Given a digraph $\mathcal{D}=\left(\mathcal{V},\mathcal{E}\right)$, a \emph{path} $\mathcal{P}$ in $\mathcal{D}$ is an ordered sequence of distinct vertices $\mathcal{P}=(v_1,\ldots,v_k)$ with $\{v_1,\dots,v_k\}\subseteq \mathcal{V}$ and $(v_{i}, v_{i+1})\in \mathcal{E}$ for all $i=1,\ldots,k-1$. %A \emph{cycle} is either a path $(v_1,\ldots,v_k)$ with the additional edge $(v_k,v_1)$ (denoted as $\mathcal{C}=(v_1,\dots,v_k,v_1)$), or a vertex with an edge to itself (i.e., self-loop, denoted as $\mathcal{C}=(v_1,v_1)$). We denote by $\mathcal{V_{C}}\subseteq\mathcal{V}$ the set of vertices in $\mathcal{C}$, and $\mathcal{E_C}\subseteq\mathcal{E}$ the set of directed edges constituting the cycle $\mathcal{C}$. The length of a cycle $\mathcal{C}$, is defined as the number of distinct vertices in $\mathcal{C}$, i.e., the cardinality of $\mathcal{V_C}$, denoted by $|\mathcal{V_C}|$. %In this paper, we only consider cycles in \textbf{directed graphs}. 
	%Given a set $\mathcal S$ of vertices in $\mathcal{D},$ we let $\mathcal{D}_{\mathcal S} = (\mathcal S,\mathcal S\times \mathcal S\subset \mathcal E)$ be the \emph{subgraph of $\mathcal{D}$ induced by} $\mathcal{S}.$ We say that $\mathcal{D}_{\mathcal S}$ can be covered by \emph{disjoint cycles} if there exists $\{\mathcal{C}_i\}_{i=1}^l$, such that $\mathcal{S}=\bigcup_{i=1}^{l}\mathcal{V}_{\mathcal{C}_i}$ and $\mathcal{V}_{\mathcal{C}_i}\cap\mathcal{V}_{\mathcal{C}_j}=\emptyset$, for all $i\ne j$, $i,j\in[l]$. 
Given a set $\mathcal{S}\subseteq\mathcal{V}$, we denote the \emph{in-neighbour set} of $\mathcal{S}$ by $\mathcal{N(S)}=\{v_i\in\mathcal{V}:(v_i,v_j)\in\mathcal{E},v_j\in\mathcal{S}\}$.% A digraph $\mathcal{D}$ is said to be \emph{strongly connected} if there exists a path between any pair of vertices. A \emph{strongly connected component} (SCC) is a maximal subgraph $\mathcal{D_S}$ of $\mathcal{D}$ such that for every $v,w\in\mathcal{S}$ there exists a path from $v$ to $w$ and from $w$ to $v$. %Given an undirected graph $\mathcal{G}=(\mathcal{V},\mathcal{E}_u)$, we consider each undirected edge $\{v_i,v_j\}\in\mathcal{E}_u$ as a pair of directed edges $\{(v_i,v_j),(v_j,v_i)\}$ if $v_i\ne v_j$, and as a self-loop $\{(v_i,v_i)\}$ otherwise. Therefore, 
	
	%Given a digraph $\mathcal{D(V,E)}$, we can use undirected graph $\mathcal{G}=(\mathcal{V},\mathcal{E}_u)$, where $\mathcal{E}_u=\{\{v_i,v_j\}\colon (v_i,v_j)\in\mathcal{E}\}$, to represent $\mathcal{D}$ when $\mathcal{E}$ satisfies symmetry constraints, i.e., $(v_i,v_j)\in\mathcal{E}$ if and only if $(v_j,v_i)\in\mathcal{E}$. 
	%
	% We say a vertex $v_i$ is \emph{reachable} from vertex $v_j$ in $G(\mathcal{V,E})$, if there exists a path from vertex $v_j$ to vertex $v_i$. 
	%	
Given a directed graph $\mathcal{D} = (\mathcal{V}, \mathcal{E})$ and two sets $\mathcal{S}_1, \mathcal{S}_2\subseteq \mathcal{V}$, we define the associated \emph{bipartite graph} of $\mathcal{D}$ by $\mathcal{B}(\mathcal{S}_1, \mathcal{S}_2, \mathcal{E}_{\mathcal{S}_1, \mathcal{S}_2}),$ whose vertex set is $\mathcal{S}_1\cup \mathcal{S}_2$ and edge set is %\footnote{We denote undirected edges using curly brackets $\{v_i,v_j\}$,in contrast with directed edges, for which we use parenthesis.}
	$\mathcal{E}_{\mathcal{S}_1, \mathcal{S}_2} = \{(s_1,s_2)\in\mathcal{E}\colon  s_1\in \mathcal{S}_1, s_2\in \mathcal{S}_2\}.$ Given $\mathcal{B}(\mathcal{S}_1, \mathcal{S}_2, \mathcal{E}_{\mathcal{S}_1, \mathcal{S}_2})$, and a set $\mathcal{S}\subseteq\mathcal{S}_1$ or $\mathcal{S}\subseteq\mathcal{S}_2$, we define the \emph{bipartite neighbor set} of $\mathcal{S}$ as $\mathcal{N_B(S)}=\{j\colon(j,i)\in\mathcal{E}_{\mathcal{S}_1,\mathcal{S}_2},i\in\mathcal{S}\}$. A \emph{matching} $\mathcal{M}$ is a set of edges in $\mathcal{E}_{\mathcal{S}_1, \mathcal{S}_2}$ that do not share vertices, i.e., given edges $e = (s_1, s_2)$ and $e^\prime = (s_1^\prime, s_2^\prime)$, $e,e^\prime \in \mathcal{M}$ only if $s_1 \neq s_1^\prime$ and $s_2 \neq s_2^\prime.$ A matching is said to be \emph{maximum} if it is a matching with the maximum number of edges among all possible matchings. Given a matching $\mathcal{M},$ two vertices $s_1$ and $s_2$ are \emph{matched} if $e=(s_1,s_2) \in \mathcal{M}.$ The vertex $v$ is said to be \emph{right-unmatched} (respectively, left-unmatched) with respect to a matching $\mathcal{M}$ associated with $\mathcal{B}(\mathcal{S}_1, \mathcal{S}_2, \mathcal{E}_{\mathcal{S}_1, \mathcal{S}_2})$ if $v \in \mathcal{S}_2$ (respectively, $v\in\mathcal{S}_1$) and $v$ does not belong to an edge in the matching $\mathcal{M}$. %We say a matching $\mathcal{M}$ is a \emph{perfect matching} if there is no right-unmatched vertex.

Given a structural pair $(\bar{A},\bar{B})$, where $\bar{A}\in\{0,\star\}^{n\times n}$ is symmetrically structured and $\bar{B}\in\{0,\star\}^{n\times m}$ is structured, we associate $(\bar{A},\bar{B})$ with a directed graph $\mathcal{D}(\bar{A},\bar{B})=(\mathcal{X\cup U},\mathcal{E_{X,X}}\cup\mathcal{E_{U,X}})$, where the vertex sets $\mathcal{X}=\{x_i\}_{i=1}^n$ and $\mathcal{U}=\{u_j\}_{j=1}^m$ are the set of state vertices and input vertices, respectively; and the edge set $\mathcal{E}_{\mathcal{X,X}}=\{(x_j,x_i)\colon [\bar{A}]_{ij}=\star\}$ and $\mathcal{E_{U,X}}=\{(u_j,x_i)\colon [\bar{B}]_{ij}=\star\}$ are the set of edges between state vertices and the set of edges between input vertices and state vertices, respectively. We also denote by $\mathcal{D}(\bar{A})=(\mathcal{X},\mathcal{E_{X,X}})$ the digraph associated with the symmetrically structured matrix $\bar{A}$ and a set $\mathcal{S}\subseteq\mathcal{X}$ is called an \emph{independence set} in $\mathcal{D}(\bar{A})$ if $\{x_i,x_j\}\notin\mathcal{E_{X,X}}$ for $\forall x_i,x_j\in\mathcal{S}$. In particular, an independence set $\mathcal{S}$ of a digraph is said to be \emph{maximal} if the number of nodes in $\mathcal{S}$ is maximal over all the independence sets of the digraph. In addition, a state vertex $x_i\in\mathcal{X}$ is said to be (input-)reachable if there exists a path from the input vertex $u_j\in\mathcal{U}$ to it. We also associate $(\bar{A},\bar{B})$ with a bipartite graph $\mathcal{B}(\bar{A},\bar{B})=(\mathcal{X\cup U,X},\mathcal{E_{X,X}}\cup\mathcal{E_{U,X}})$, which we refer to as the \emph{system bipartite graph}.

%Given a symmetrically structured matrix $\bar{A},$ we associate it with a directed graph $\mathcal{D}(\bar{A})=(\mathcal{X},\mathcal{E}(\bar{A})),$ which we refer to as the \emph{state digraph}, where $\mathcal{X}=\{x_i\}_{i=1}^n$ is the set of state vertices, and $\mathcal{E}(\bar{A})=\{(x_j,x_i):[\bar{A}]_{ij} = \star\}$ is the set of directed edges. To capture the symmetrical parameter dependencies, it is also useful to associate with $\bar{A}$ an undirected graph $\mathcal{G}(\bar{A})=(\mathcal{X},\mathcal{E}_u(\bar{A}))$, where $\mathcal{E}_u(\bar{A})=\{\{x_i,x_j\}\colon [\bar{A}]_{ij}=\star,i\le j\}$ is the set of undirected edges. Similarly, we associate with the structural pair $(\bar{A},\bar{B})$ a directed graph $\mathcal{D}(\bar{A},\bar{B})=(\mathcal{X}\cup\mathcal{U}, \mathcal{E}(\bar{A})\cup \mathcal{E}_{\mathcal{U},\mathcal{X}})$, which we refer to \emph{system digraph}, where $\mathcal{U}=\{u_i\}_{i=1}^m$ is the set of input vertices and $\mathcal{E}_{\mathcal{U},\mathcal{X}} = \{(u_j, x_i)\colon [\bar{B}]_{ij} =\star\}$ is the set of edges from input vertices to state vertices. Due to symmetrical parameter constraints in $A$, we also associate with $(\bar{A},\bar{B})$ a mixed graph, referred to as the \emph{system mixed graph}, $\mathcal{G}(\bar{A},\bar{B})=\{\mathcal{X\cup U},\mathcal{E}_u(\bar{A}),\mathcal{E_{U,X}}\}$ containing undirected edges between state vertices and directed edges from input vertices to state vertices.
\section{Analysis of Structural Stabilizability}\label{sec:Results}
In what follows, we have two subsections where we address Problems~\ref{p0} and \ref{p1}. Specifically, in Section~\ref{sec:graph condition}, we obtain Theorem~\ref{con_SS} that characterizes the solutions to Problem~\ref{p0}, whereas in Section~\ref{sec: stabilizable}, Theorem~\ref{dim} gives a characterization of the maximum dimension of stabilizable subspace, which addresses Problem~\ref{p1}.
%In this section, we first propose a graph-theoretic necessary and sufficient condition for structural stabilizability % of a structural pair involving symmetrically structured matrix in Theorem~\ref{con_SS}. By leveraging the proposed conditions, we obtain a solution to Problem~\ref{p1} that we present in Theorem~\ref{dim}. 

\subsection{Graph-Theoretic Conditions on Structural Stabilizability}\label{sec:graph condition}
%In order to derive a necessary and sufficient condition for structural stabilizability, we first investigate this problem by an algebraic approach, then we interprete those algebraic conditions by equivalent graph-theoretic conditions, which can be verified efficiently by existing algorithms. 
Since the stabilizability concerns the stability of the uncontrollable part of $(A, B),$ it is necessary to first characterize the controllable and uncontrollable parts from the structural information contained in the pair $(\bar{A},\bar{B})$. We recall a lemma from \cite{li2018tc} that characterizes controllable modes for the numerical realizations of a structural pair.
\begin{lem}[\cite{li2018tc}]\label{lem:proper}
	Given a structural pair $(\bar{A},\bar{B})$, where $\bar{A}\in\{0,\star\}^{n\times n}$ is symmetrically structured, and $\textrm{t--rank}(\bar{A})=k$, if $(\bar{A},\bar{B})$ is irreducible, then there exists a proper variety $V\subset \mathbb{R}^{n_{\bar{A}}+n_{\bar{B}}}$, such that for any numerical realization $(\tilde{A},\tilde{B})$ with $[\mathbf{p}_{\tilde{A}},\mathbf{p}_{\tilde{B}}]\in V^c$, $\tilde{A}$ has $k$ nonzero, simple and controllable modes.
\end{lem}

Lemma~\ref{lem:proper} shows that the irreducibility of $(\bar{A},\bar{B})$ guarantees that all the non-zero modes of $(\tilde{A},\tilde{B})$ are controllable generically. Subsequently, %by a contrapositive argument, 
we can claim that given an irreducible pair $(\bar{A},\bar{B})$, if for any numerical realization $(\tilde{A},\tilde{B})$ there exists an uncontrollable eigenvalue, then that uncontrollable eigenvalue is $0.$ This implies that $(\tilde{A},\tilde{B})$ is not stabilizable. Therefore, if a pair $(\bar{A},\bar{B})$ is irreducible but not structurally controllable, then $(\bar{A},\bar{B})$ is not structurally stabilizable. Hence, we have the following lemma.
\begin{lem}\label{lem_2}
	Given an irreducible structural pair $(\bar{A},\bar{B})$, where $\bar{A}\in\{0,\star\}^{n\times n}$ is symmetrically structured, then $(\bar{A},\bar{B})$ is structurally stabilizable if and only if $(\bar{A},\bar{B})$ is structurally controllable.%we associate the digraph $\mathcal{D}(\bar{A},\bar{B})$ with the pair $(\bar{A},\bar{B})$. %If there exists a subset $\mathcal{S}\subseteq\mathcal{X}$ such that $\left|\mathcal{N(S)}\right|< \left|\mathcal{S}\right|$, then there does not exist a numerical realizaiton $(\tilde{A},\tilde{B})$ such that $(\tilde{A},\tilde{B})$ is stabilizable.
\end{lem}
While Lemma~\ref{lem_2} is a condition for structural stabilizability when $(\bar{A},\bar{B})$ is irreducible, we should also consider the case when $(\bar{A},\bar{B})$ is reducible. By the definition of reducibility, $(\bar{A},\bar{B})$ can be permuted to the form of~\eqref{eq}. %Since $\bar{A}$ is symmetric, it results that $\bar{A}_{22}$ is also symmetric. Meanwhile, 
In order for $(\bar{A}, \bar{B})$ to be structurally stabilizable, it is required that there exists a numerical realization $\tilde{A}_{22}$ whose eigenvalues of are all negative. Summarizing these two arguments, it is equivalent to say that whether there exists a negative definite numerical realization $\tilde{A}_{22}$ determines whether the structural pair is stabilizable. Consequently, it is important to determine when the above claim is true, as follows.

%We then present a condition which guanrantees the structural stabilizability when $\bar{B}=\mathbf{0}$.}
\begin{lem}\label{lem_1}
Given a reducible structural pair $(\bar{A},\bar{B})$, where $\bar{A}\in\{0,\star\}^{n\times n}$ is in the form of ~\eqref{eq}. Then there exists a numerical realization $\tilde{A}_{22}$ which is negative definite if and only if the diagonal entries of $\bar{A}_{22}$ are all $\star$-entries.
\end{lem}

Combining Lemmas~\ref{lem_2} and~\ref{lem_1}, we have an algebraic condition for structurally stabilizability. In what follows, we present a graph-theoretic interpretation of these conditions. %We interpret those conditions as graph-theoretic ones.
\begin{thm}%[\textbf{Graph-Theoretic Necessary and Sufficient Condition for Structural Stabilizability}]
	\label{con_SS}
	Consider a structural pair $(\bar{A},\bar{B})$, where $\bar{A}$ is symmetrically structured. Let $\mathcal{D}(\bar{A},\bar{B})=(\mathcal{X}\cup\mathcal{U},\mathcal{E_{X,X}}\cup\mathcal{E_{U,X}})$ be the digraph associated with $(\bar{A},\bar{B})$, and $\mathcal{X}_r\subseteq\mathcal{X}$ and $\mathcal{X}_u\subseteq\mathcal{X}$ be the subset of state vertices which are input-reachable and input-unreachable, respectively. The $(\bar{A},\bar{B})$ is structurally stabilizable if and only if the following two conditions hold simultaneously in $\mathcal{D}(\bar{A},\bar{B})$:
	\begin{enumerate}
		\item the vertex $x_i$ has a self-loop, $\forall x_i\in\mathcal{X}_u$;
		\item $|\mathcal{N}(\mathcal{S})|\ge|\mathcal{S}|$, $\forall \mathcal{S}\subseteq \mathcal{X}_r$.
%		\item for $\forall \mathcal{S} \subseteq\mathcal{X}$, $\left|\mathcal{N(S)}\right|\ge\left|\mathcal{S}\right|$.
	\end{enumerate}
\end{thm}
Essentially, to ensure structural stabilizability, two conditions should hold simultaneously: (\emph{i}) every unreachable state vertex should have a self-loop, and (\emph{ii}) the reachable part of the system should be structurally controllable~\cite{li2018tc}. 
%hese graph-theoretic conditions can be verified in polynomial time and the attacker can use them to destroy the stabilizability of the network with only information of network topology, which will be discussed in the subsection \ref{sec3}. Before that, 

Next, we utilize Theorem~1 to characterize the maximum dimension of the stabilizable subspace.

\subsection{Maximum Dimension of Stabilizable Subspace}\label{sec: stabilizable}
Similar to the previous subsection, we will first consider the case when $(\bar{A},\bar{B})$ is irreducible, then extend the solution approach to the general case.

By Lemma~\ref{lem_2}, when $(\bar{A},\bar{B})$ is irreducible, the $(\bar{A},\bar{B})$ is structurally controllable if and only if it is structurally stabilizable. This motivates us to consider the relationship between controllable subspace and stabilizable subspace. Moreover, it is shown in \cite{hosoe1980determination} that the maximum dimension of controllable subspace is equal to the generic dimension of controllable subspace of a structural pair without symmetric parameter constraints. We may suspect that equality also holds when symmetric parameter dependency is considered. Motivated by this intuition, we first study the generic dimension of the controllable subspace, and then extend the derived results to obtain a solution of Problem~\ref{p1}.
%By definition, stabilizability is highly related with controllability. Therefore, we first determine the generic dimension of controllable subsapce, and then extend the derived results to a solution of Problem~\ref{p1}. 

Given a structured pair $(\bar{A},\bar{B})$, where $\bar{A}$ is symmetrically structured, if there exists a proper variety $V\subset\mathbb{R}^{n_{\bar{A}}+n_{\bar{B}}}$, such that $\textrm{rank}(Q(\tilde{A},\tilde{B}))=k$ when $[\mathbf{p}_{\tilde{A}},\mathbf{p}_{\tilde{B}}]\in V^c$, then we say the \emph{generic dimension} \cite{hosoe1980determination} of controllable subspace of $(\bar{A},\bar{B})$, denoted as $d_c$, is $k$. For almost all numerical realizations $(\tilde{A},\tilde{B})$ with $[\mathbf{p}_{\tilde{A}},\mathbf{p}_{\tilde{B}}]\in\mathbb{R}^{n_{\bar{A}}+n_{\bar{B}}}$ (except for a proper variety, e.g., $[\mathbf{p}_{\tilde{A}},\mathbf{p}_{\tilde{B}}]\in V$), the dimension of controllable subspace is $d_c$. 

We characterize the generic dimension of controllable subspace of a structural pair involving a symmetrically structured matrix by the following lemma.
%\color{red}
	\begin{lem}\label{theo6}
		Given an irreducible structural pair $(\bar{A},\bar{B})$, where $\bar{A}\in\{0,\star\}^{n\times n}$ is symmetrically structured and $\bar{B}\in\{0,\star\}^{n\times m}$ is structured, the generic dimension of controllable subspace equals to the term rank of $[\bar{A},\bar{B}]$, i.e., the concatenation of matrices $\bar{A}$  and~$\bar{B}$.%, which is the concatenation of matrices $\bar{A}$ and $\bar{B}$.%$\textrm{t--rank}([\bar{A},\bar{B}])$.
\end{lem}

When $(\bar{A},\bar{B})$ is reducible, we can permute $(\bar{A},\bar{B})$ to obtain the form in \eqref{eq}. By Definition~\ref{dss} and Theorem~\ref{con_SS}, the maximum dimension of the stabilizable subspace should be the sum of the generic dimension of controllable subspace and the maximum number of negative eigenvalues over all the numerical realizations of the uncontrollable part. This can be formalized in the following result. % this reasoning as the following Theorem.\vspace{-0.1cm}

%We denote by $\textrm{m-dims}(\bar{A},\bar{B})$ as the maximum dimension of stabilizable subspace of numerical realizations of the structural pair $(\bar{A},\bar{B})$, i.e., there exists a numerical realization $(\tilde{A},\tilde{B})$ such that there are $\textrm{m--dims}(\bar{A},\bar{B})$ states that can be stabilized.
\begin{thm}\footnote{Please note that Theorem 2 in the ACC version of this paper \cite{li2019acc} is not correct and has been revised here.}%[\textbf{Maximum Dimension of Stabilizable Subspace}]
	\label{dim}
	Consider a structural pair $(\bar{A},\bar{B})$, where $\bar{A}\in\{0,\star\}^{n\times n}$ is symmetrically structured. Then,
	\begin{enumerate}
		\item if $(\bar{A},\bar{B})$ is irreducible, then the maximum dimension of stabilizable subspace of $(\bar{A},\bar{B})$ equals to the generic dimension of controllable subspace of $(\bar{A},\bar{B})$;
		\item {\color{black}if $(\bar{A},\bar{B})$ is reducible, then we permute the matrix $\bar{A}$ into the form~\eqref{eq}. %The $\textrm{m-dim}(\bar{A},\bar{B})$ is {\color{red}larger or equal to} 
		Let $\{\mathcal{C}_i\}_{i=1}^{k}$ be a set of disjoint cycles in $\mathcal{D}(\bar{A}_{22})$. On one hand, we have
		\begin{equation}
		\begin{aligned}
		\emph{m-dim}(\bar{A},\bar{B})\ge&\emph{t--rank}([\bar{A}_{11},\bar{B}_{1}])+\dfrac{1}{2}\sum_{i\in\mathcal{S}_e}|\mathcal{C}_i|\\& +\dfrac{1}{2}\sum_{i\in\mathcal{S}_o}\left(|\mathcal{C}_i|+1\right),
		\end{aligned}
		\end{equation}
		
		where $\mathcal{S}_e$ is the set of indexes of cycles with even length in $\{\mathcal{C}_i\}_{i=1}^k$and $\mathcal{S}_o$ is the set of indexes of cycles with odd length in $\{\mathcal{C}_i\}_{i=1}^k$; On the other hand, we have
		\begin{equation}
			\emph{m-dim}(\bar{A},\bar{B})\le \emph{t-rank}([\bar{A}_{11},\bar{B}_{1}])+(n-p)-|\mathcal{S}|,
		\end{equation}
		where $\mathcal{S}$ is a maximal independence set in $\mathcal{D}(\bar{A}_{22})$ and $\bar{A}_{22}\in\{0,\star\}^{(n-p)\times(n-p)}$.}
		%The $\textrm{m-dim}(\bar{A},\bar{B})$ is {\color{red}larger or equal to} $\textrm{t--rank}([\bar{A}_{11},\bar{B}_{1}])+k$, where $k$ is the total number of $\star$-entries in the diagonal of~$\bar{A}_{22}$.
	\end{enumerate}
\end{thm}

	\begin{remark}
	In the form~\eqref{eq}, the index of columns of $\bar{A}_{11}$ are corresponding to input-reachable state vertices in $\mathcal{D}(\bar{A},\bar{B})$, and the index of columns of $\bar{A}_{22}$ are corresponding to the input-unreachable state vertices in $\mathcal{D}(\bar{A},\bar{B})$. The input-reachable/unreachable vertices can be identified by running a depth-first search \cite{awerbuch1985new}. Besides, the term-rank of $([\bar{A}_{11},\bar{B}_{1}])$ can be obtained by finding a maximum bipartite matching in $\mathcal{B}(\bar{A},\bar{B})$ \cite{li2018tc}.
\end{remark}
\section{Optimal Actuator-Attack and Recovery Problems}\label{sec:Design}
In this section, equipped with the results from Section~\ref{sec:Results}, we show the NP-hardness of Problem~\ref{p2} and Problem~\ref{p3} in Theorem~\ref{nphard} and Theorem~\ref{nphard2}, respectively. Then, we introduced a greedy algorithm to solve Problem~\ref{p3} -- see Algorithm~\ref{greedy}. Besides, we show that Algorithm~\ref{greedy} achieves a $(1-1/e)$ approximation guarantee to the optimal solution of Problem~\ref{p3}, which is formally captured in Theorem~\ref{them_alg}.
\subsection{Computational Complexity of Problem~\ref{p2}}\label{sec3}
%In this subsection, we characterize the computational complexity of Problem~\ref{p2}.
Suppose that there is no self-loop in the system digraph of a structural pair $(\bar{A},\bar{B})$ and the Condition-2) in Theorem~\ref{con_SS} is satisfied. Then, we will show that Problem~\ref{p2} is equivalent to minimizing the number of input-reachable states by removing a limited number of inputs. %Reachibility is closely related to the a set cover problem in combinotoric mathematics \cite{pequito2015complexity}. 
%In this case, Problem~\ref{p2} becomes to delete a limited number of actuators to minimize the total number of reachable state vertices, 
This problem shares the similarities with Min-k-Union problem described next.%(given below), in which we aim at selecting a limited number of sets such that the union of these sets is minimized. \vspace{-0.2cm}
\begin{define}[Min-k-Union Problem~\cite{vinterbo2002note}]\label{df:mku}
	Given a universe $\mathcal{U_S}=\{\mathcal{S}_{\ell}\}_{\ell=1}^p$ and an integer $k\in\mathbb{Z}^+$, find 
	\begin{equation}\small
	\begin{split}
	\mathcal{L^*}=&\arg \min_{\mathcal{L}=\{\ell_i\}_{i=1}^k} |\bigcup_{i=1}^k\mathcal{S}_{\ell_i}|\\
	&\emph{ s.t. }\  \mathcal{L}\subseteq [p].
	\end{split}
	\end{equation}
\end{define}

%Suppose the Condition-2) is satisfied, then Problem~\ref{p2} Therefore, 
Therefore, we aim at selecting a limited number of sets whose union is minimized, leading to the following result.%The above reasoning leads to the following theorem.\vspace{-0.1cm}
\begin{thm}\label{nphard}
The Optimal Actuator-disabling Attack Problem (Problem~\ref{p2}) is NP-hard.% in general.\vspace{-0.1cm}
\end{thm}

Although the problem is NP-hard, that does not imply that all instances of the problem are equally difficult. As a consequence, we now propose to characterize the approximability of Problem~\ref{p2}. %provide a partial polynomial reduction from Problem~\ref{p2} to the Min-k-Union problem under some assumptions on the topology of $\mathcal{G}(\bar{A})$. %we can reduce Problem~\ref{p2} to the min-k-union problem. 
We first consider a subclass of instances of Problem~\ref{p2}, which satisfy the following assumption.%derived by imposing some assumption on the structural pair $(\bar{A},\bar{B})$, as described next.\vspace{-0.1cm}
\begin{assumption}\label{asm:perfect matching}{\color{black}
	The symmetrically structured matrix $\bar{A}\in\{0,\star\}^{n\times n}$ is such that for any $\mathcal{S\subseteq X}$, where $\mathcal{X}$ is the set of state vertices in the state digraph $\mathcal{D}(\bar{A})$, $|\mathcal{N(S)}|\ge|\mathcal{S}|$.}% In addition, there exists no vertex with self-loop in $\mathcal{D}(\bar{A})$.
\end{assumption}

Assumption~\ref{asm:perfect matching} ensures that in the bipartite graph associated with $\mathcal{D}(\bar{A})$, there is no right-unmatched vertex with respect to any maximum matching, i.e., the Condition-2) in Theorem~\ref{con_SS} is always satisfied. %In addition, by Assumption~\ref{asm:perfect matching}, it follows that the diagonal entries of $\bar{A}$ satisfy $[\bar{A}]_{ii}=0$, for $\forall i\in[n]$. 
We then have the following theorem.

%Having shown that the min-k-union problem can be reduced to the Problem~\ref{p2} in Theorem~\ref{nphard}, we then show that under Assumption~\ref{asm:perfect matching}, the Problem~\ref{p2} can be reduced to min-k-union problem in the following theorem.
%To ensure the structural stabilizability, we need to find a set of actuators such that Condition-1) in Theorem~\ref{con_SS} holds. Therefore, we 

%Recalling the similarity between the problem of selecting actuators to achieve reachability of state vertices and the set cover problem, we have the following partial polynomial reduction from Problem~\ref{p2} to min-k-union problem.

\begin{thm}\label{thm:partial reduction}
{\color{black}Under Assumption~\ref{asm:perfect matching}, denote by $m_1$ the total number of sets (i.e., $\{\mathcal{S}_i\}_{i=1}^{m_1}$) in an instance of Min-k-Union problem, and $m_2$ the total number of candidate inputs in an instance of Problem~\ref{p2}. Additionally, let $\rho:\mathbb{Z}\rightarrow \mathbb{R}$. Then, there exists a $\rho(m_2)$-approximation algorithm for Problem~\ref{p2} if there exists a $\rho(m_1)$-approximation algorithm for Min-k-Union problem.}
\end{thm}
%\begin{thm}\label{thm:partial reduction}
%Under the Assumption~\ref{asm:perfect matching}, Problem~\ref{p2} can be polynomially reduced to the Min-k-Union problem.
%\end{thm}

{\color{black}As a result of Theorem \ref{thm:partial reduction}, any approximation algorithm solving Min-k-Union problem can be adapted to solve Problem~\ref{p2} with approximation guarantees. }%Problem~\ref{p2} is at least as hard as the Min-k-Union problem.

\subsection{Solution to Problem~\ref{p3}}\label{sec4}
To investigate the computation complexity of obtaining a solution to Problem~\ref{p3}, we take a similar strategy to that used in the previous section, i.e., we first consider the following special instance: the pair $(\bar{A},\bar{B})$ satisfies Assumption~1. In this case, we will show that Problem~\ref{p3} is equivalent to adding a limited number of actuators to maximize the total number of input-reachable state vertices, which is similar to the Max-k-Union problem, stated as follows.% as follows.% similar to the reasoning taken in the Theorem~\ref{nphard}, the problem of adding . We can show the NP-hardness of Problem~\ref{p3} by reducing the max-k-union problem to an instance of Problem~\ref{p3}. We first recall the definition of Max-k-union problem.
\begin{define}[Max-k-Union Problem \cite{hochbaum1997approximating}]
	Given a universe $\mathcal{U_S}=\{\mathcal{S}_{\ell}\}_{\ell=1}^p$ and an integer $k\in\mathbb{Z}^+$, find
	\begin{equation}\small
	\begin{split}
		\mathcal{L}^*=&\arg \max_{\mathcal{L}=\{\ell_i\}_{i=1}^k} |\bigcup_{i=1}^k \mathcal{S}_{\ell_i}|\\
		&\emph{ s.t. }\mathcal{L}\subseteq[p].
	\end{split}
	\end{equation}
\end{define}
%The NP-hardness of Problem~\ref{p3} is shown by the following theorem.
Thus, we obtain the following theorem.
\begin{thm}\label{nphard2}
	The Optimal Recovery Problem (Problem~\ref{p3}) is NP-hard.
\end{thm}
%{\color{blue} THEN WHY CANT YOU USE GREEDY ALGORITHMS IN THE PROBLEM 3??}

A natural approximation solution to optimal design problems is through greedy algorithms~\cite{fujishige2005submodular}. Although greedy algorithms may not provide an optimal solution, under specific objective functions of the problem, a suboptimal solution with suboptimally guarantees can be provided. Specifically, a particular class of problem with such properties is called submodularity function problems, defined as follows. %More specifically, if a function is called \emph{sub-modular} if {\color{blue} DEFINITION}
\begin{define}[Submodular function \cite{fujishige2005submodular}]
	Let $\Omega$ be a nonempty finite set. A set function $f\colon 2^{\Omega}\to \mathbb{R}$, where $2^{\Omega}$ denotes the power set of $\Omega$, is a \emph{submodular function} if for every $\mathcal{J}_1,\mathcal{J}_2\subseteq\Omega$ with $\mathcal{J}_1\subseteq \mathcal{J}_2$ and every $i\in\Omega\setminus\mathcal{J}_2$, we have $f(\mathcal{J}_2\cup\{i\})-f(\mathcal{J}_2)\le f(\mathcal{J}_1\cup\{i\})-f(\mathcal{J}_1)$.% In addition, a set function $f\colon 2^{\Omega}$ is a \emph{sup-modular function} if $-f$ is submodular.
\end{define}
The greedy algorithm \cite{fujishige2005submodular} achieves a $(1-1/e)$-factor approximation to the optimal solution provided that the objective function is submodular. In this paper, we show that the objective function in Problem~\ref{p3} is submodular; hence, the greedy algorithm provides a constant factor guarantee to the optimal solutions.
\begin{algorithm}[t]
	\caption{$(1-1/e)$ approximation solution to Problem~\ref{p3}}	\label{greedy}
	\small
	\begin{algorithmic}[1]
		\Require The pair $(\bar{A},\bar{B})$, $\bar{B}_{\mathcal{U}_{can}}\in\{0,\star\}^{n\times m'}$, and the budget $k$;
		\Ensure Suboptimal solution $\mathcal{J}$;
%		\For {each $u_i\in\mathcal{U}_{can}$}
%		\State  $d_i\gets\textrm{g--dims}(\bar{A},\bar{B}_{\mathcal{U}\cup\{u_i\}})$;
%		\EndFor
%		\State Permute $u_i\in\mathcal{U}_{can}$ such that $d_i\ge d_{i+1},\forall i\le(|\mathcal{U}_{can}|-1)$;
		\State Initialize $\mathcal{J}\gets\emptyset$, $\mathcal{L}\gets [m']$;\Comment{$\mathcal{L}$ is the set of indexes of new actuators in $\mathcal{U}_{can}$, the set of new actuators that can be added to the system.}
		\For {iteration $i\in [k]$}
		\For {each $j\in\mathcal{L}$}
		\State $d_j\gets \textrm{m-dim}(\bar{A},[\bar{B},\bar{B}_{can}(\mathcal{J}\cup\{j\})])$;
		\EndFor
		\State $\mathcal{I}\gets \{i\colon d_{i}=\max\{d_j\}_{j=1}^{|\mathcal{L}|}\}$;
		\State Pick a $j\in\mathcal{I}$;
		\State $\mathcal{J}\gets \mathcal{J}\cup\{j\}$;
		\State $\mathcal{L}\gets\mathcal{L}\setminus\{j\}$;
		\EndFor\\
%		\State $\mathcal{\hat{U}}_a\gets\mathcal{U}_a$;\\
		\Return $\mathcal{J}$
	\end{algorithmic}
\end{algorithm}

%It has been shown in~\cite{feige1998threshold} that under the assumption P$\ne$NP, the greedy algorithm, i.e., at each iteration the algorithm choose the set containing the largest number of uncovered elements, is the best-possible polynomial-time approximation algorithm for the maximum k-union problem. Through the proof of Theorem~\ref{nphard2}, we see that solving Problem~\ref{p3} is at least as hard as the maximum k-union problem. We next show that the Algorithm~\ref{greedy}, which generalizes the greedy idea discussed above, is the $(1-1/e)$ approximation solution to the Problem~\ref{p3}.

\begin{thm}\label{them_alg}
Algorithm~\ref{greedy} returns a $(1-1/e)$-approximation of the optimal solution to Problem~\ref{p3}.
\end{thm}
\begin{remark}\label{rem}
In \cite{chlamtac2018densest}, the authors argue that insofar there is no constant factor approximation to the Min-k-Union problem. Thus, together with Theorem~\ref{thm:partial reduction}, we cannot use the greedy algorithm to approximate Problem~\ref{p2} with guarantee. % it is hard to find a better than polynomial approximation algorithm for Min-k-Union problem. In contrast to Problem~\ref{p3}, it is hard to find a constant factor approximation algorithm for Problem~\ref{p2}.
\end{remark}
\section{Illustrative Examples}\label{sec:Examples}
In this section, we present examples to illustrate our results on structural stabilizability and approximation solution to Problem~\ref{p3}.
\subsection{Maximum Dimension of Stabilizable Subspace}
We consider a structural pair $(\bar{A},\bar{B})$, where $\bar{A}\in\{0,\star\}^{11\times 11}$ is symmetrically structured and $\bar{B}\in\{0,\star\}^{11\times 1}$ is structured. 
\begin{equation}\footnotesize\label{eq:ex}
	\bar{A}=\left[\begin{smallmatrix}
	     0&a_{12}&     0&a_{14}& a_{15}&    0&     0&0&0&0&0\\
	a_{12}&     0&     0&     0&     0&     0&     0&0&0&0&0\\
	     0&	    0&     0&     0&     0&a_{36}&a_{37}&0&0&0&0\\
	a_{14}&     0&     0&     0&     0&     0&     0&0&0&0&0\\
	a_{15}&     0&     0&     0&0&     0&     0&0&0&0&0\\
		 0&     0&a_{36}&     0&     0&a_{66}&     0&0&0&0&0\\
		 0&     0&a_{37}&     0&     0&     0&     0&0&0&0&0\\
		 0&     0&     0&     0&     0&     0&     0&0&   0& a_{810}&0\\
		 0&0&0&0&0&0&0&0&a_{99}&a_{910}&0\\
		 0&0&0&0&0&0&0&a_{810}&a_{910}&0&a_{1011}\\
		 0&0&0&0&0&0&0&0&0&a_{1011}&0\\
	\end{smallmatrix}\right], \bar{B}=\left[\begin{smallmatrix}
	b_{11}\\0\\0\\b_{41}\\0\\0\\0\\0\\0\\0\\0
	\end{smallmatrix}\right].
\end{equation}
We depict the digraph representation of the structural pair $(\bar{A},\bar{B}),$ denoted by $\mathcal{D}(\bar{A},\bar{B}),$ in Figure~\ref{fig2}. Since $x_3$ and $x_7$ are unreachable vertices and they do not have self-loops, the pair $(\bar{A},\bar{B})$ is not structurally stabilizable due to Theorem~\ref{con_SS}. Furthermore, the total number of right-matched (with respect to any maximum matching in the associated bipartite graph $\mathcal{B}(\bar{A},\bar{B})$) reachable vertices is $3$. On one hand, a maximal independence set in the input-unreachable part has a cardinality of $3$. By Theorem~\ref{dim}, $\textrm{m-dim}(\bar{A},\bar{B})\le 3+7-3=7$. On the other hand, there are two vertex-disjoint cycles with length 2 and two cycles with length 1 in the input-unreachable part. By Theorem~\ref{dim}, $\textrm{m-dim}(\bar{A},\bar{B})\ge 3+2+2=7$. Therefore, by invoking Theorem~\ref{dim}, we conclude that the maximum stabilizable subspace is $7$. 
\begin{figure}[t]
\centering
\includegraphics[width=0.36\textwidth]{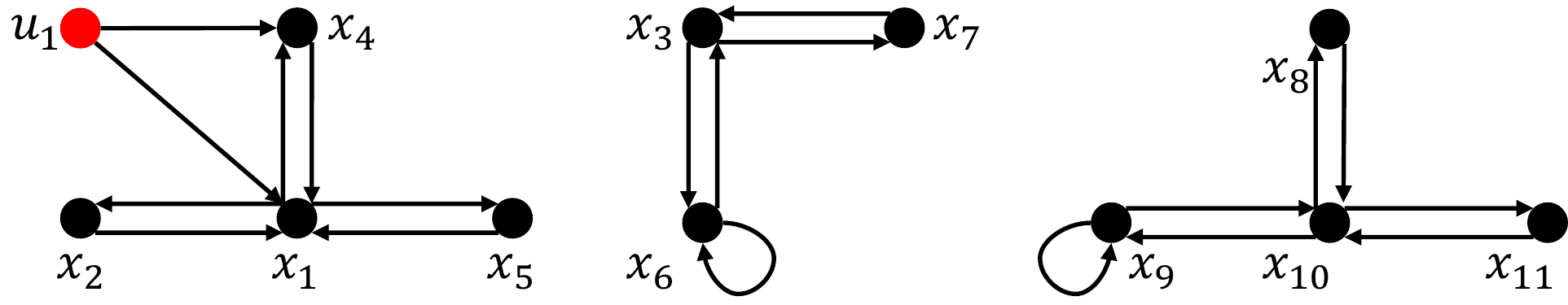}%\vspace{-0.1cm}
\caption{In this figure, we depict the structure of $\mathcal{D}(\bar{A},\bar{B}).$ The red vertex labeled by $u_1$ and black vertices labeled by $x_1, \ldots, x_{11}$ are the input vertex and state vertices, respectively. The black arrows represent the edges from input vertex to state vertices, as well as edges between state vertices.}\label{fig2}
\end{figure}
\subsection{Optimal Recovery Problem}
%{\color{blue}Since we have shown that Problem~\ref{p2} is }
Now, we present an example to illustrate the use of Algorithm~\ref{greedy}. Consider again the structural pair $(\bar{A},\bar{B})$ specified in \eqref{eq:ex}. As noted in the last subsection, the $(\bar{A},\bar{B})$ is not structurally stabilizable. We let $\mathcal{U}_{can}=\{u_i\}_{i=2}^7$ be the set of candidate actuators that can be added into the system and associate it with the structured matrix $\bar{B}_{\mathcal{U}_{can}}\in\{0,\star\}^{11\times 6}$, of which nonzero entries are captured by the red edges of the digraph $\mathcal{D}(\bar{A},[\bar{B},\bar{B}_{\mathcal{U}_{can}}])$ depicted in Figure~\ref{fig3}.%, encoding the interconnection among state vertices and input v.% We depict the edges in the set $\mathcal{E}(\bar{B}_{\mathcal{U}_{can}})$ in Figure~\ref{fig3}.

We have obtained in the last subsection that $\textrm{m-dim}(\bar{A},\bar{B})$ is at least $5$. Suppose we have a budget $k=3$, then Problem~\ref{p3} consists in adding 3 actuators from $\mathcal{U}_{can}$ into the system such that the maximum stabilizable subspace is maximized. In the first iteration of Algorithm~\ref{greedy}, $u_4$ is selected because $\textrm{m-dim}(\bar{A},[\bar{B},\bar{B}_{\mathcal{U}_{can}}(\{4\})])-\textrm{m-dim}(\bar{A},\bar{B})=2\ge \textrm{m-dim}(\bar{A},[\bar{B},\bar{B}_{\mathcal{U}_{can}}(\{i\})])-\textrm{m-dim}(\bar{A},\bar{B}), \forall u_i\in\mathcal{U}_{can}.$ Similarly, in the second iteration, $u_3$ is selected by Algorithm~\ref{greedy}. This results that $\textrm{m-dim}(\bar{A},[\bar{B},\bar{B}_{\mathcal{U}_{can}}(\{3,4\})])=10$. Finally, $u_{7}$ is selected and $\textrm{m-dim}(\bar{A},[\bar{B},\bar{B}_{\mathcal{U}_{can}}(\{3,4,7\})])=11$. Since the maximum possible stabilizable subspace is always less than or equal to the total number of states, in this example, Algorithm~\ref{greedy} returns an optimal solution to Problem~\ref{p3}.

\begin{figure}[t]
	\centering
	\includegraphics[width=0.36\textwidth]{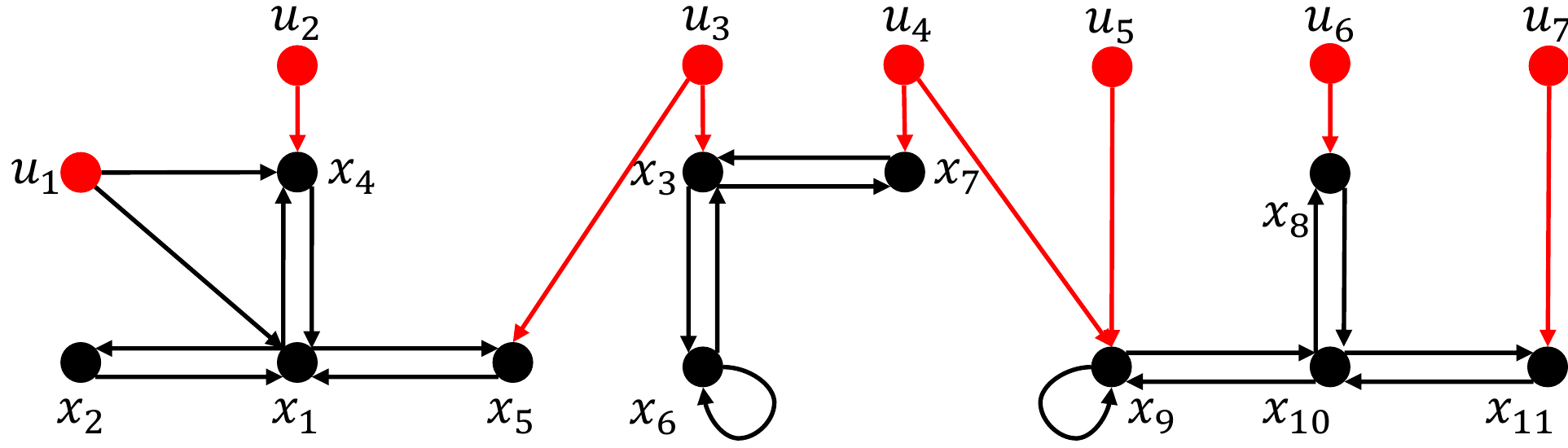}
\caption{In this figure, we depict the digraph $\mathcal{D}(\bar{A},[\bar{B},\bar{B}_{\mathcal{U}_{can}}]).$ We use red and black vertices to represent input vertices and state vertices, respectively. The black and red arrows represent are the edges in $\mathcal{E_{X,X}}\cup\mathcal{E}_{\{u_1\},\mathcal{X}}$ and edges in $\mathcal{E}_{\mathcal{U}_{can},\mathcal{X}}$, respectively.}\label{fig3}
\end{figure}
\section{Conclusion}\label{sec:conclusion}
In this paper, we studied the structural stabilizability problem of undirected networked dynamical systems. We proposed a graph-theoretic method to derive the maximum dimension of stabilizable subspace of an undirected network. In addition, we formulated the optimal actuator-disabling attack problem and optimal recovery problem. We proved that these two problems are NP-hard. Finally, we developed a $(1-1/e)$ approximation algorithm for the optimal recovery problem. 

In the future, we will focus on developing approximation algorithms for the optimal actuator-disabling attack problem. Furthermore, it would also be of interest to relax the assumption on symmetricity, and extend the results in this paper to directed networked systems.
\section*{Appendix}
%\subsection{Proofs of Theorem~\ref{con_SS} and related results}

\begin{proof}[Proof of Lemma~\ref{lem_2}]
First, we notice that sufficiency follows from the fact that structural controllability ensures that almost surely there exists numerical realization ensuring controllability, which implies that any desired state can be attained by a finite sequence of inputs. Therefore, if there was not one such sequence, then the uncontrollable subspace is nonempty, and the only way to ensure that we can take the state to the origin is when the subspace is stable. a control input driving the states to the origin in finite time. Necessity follows by contrapositive argument. %Here we adopt all the notations taken in the proof of Theorem 4 in \cite{li2018tc}. 
Suppose $(\bar{A},\bar{B})$ is irreducible but not structurally controllable, then by Theorem 1 in \cite{li2018tc}, there exists a set $\mathcal{S\subseteq X}$ such that $|\mathcal{N(S)}|<|\mathcal{S}|$, which implies that $\textrm{g-rank}([\bar{A},\bar{B}])<n$. %Suppose $\exists \mathcal{S}\subseteq\mathcal{T}$, such that $\left|\mathcal{N(S)}\right|<\left|\mathcal{S}\right|$, then by \cite[Lemma~]{li2018tc} $\textrm{rank}([\tilde{A},\tilde{B}])<n$. 
For $\forall [\mathbf{p}_{\tilde{A}},\mathbf{p}_{\tilde{B}}]\in\mathbb{R}^{n_{\bar{A}}+ n_{\bar{B}}}$, $\exists v\in\mathbb{C}^n$, such that $v^T[\tilde{A},\tilde{B}]=0$, i.e., $v^T\tilde{A}=v^T0$. Consequently, there exists a zero eigenvalue which is not controllable, hence not stabilizable.
\end{proof}

\begin{proof}[Proof of Lemma~\ref{lem_1}]
%	We first prove the sufficiency $(\Longleftarrow)$, and then prove the necessity $(\Longrightarrow)$ of the claim.
%
\emph{(If)} Let us construct a numerical realization $\tilde{A}_{22}$ by assigning zero value to off-diagonal $\star$-entries of $\bar{A}_{22}$, and negative values to $\star$-entries on the diagonal. In this case, matrix $\tilde{A}_{22}$ is negative definite diagonal matrix.
	
\emph{(Only if)} We approach the proof by contrapositive. Let $m$ be the dimension of $\bar{
A}_{22}, $ and $\{v_i\}_{i=1}^m$ be the standard basis in $\mathbb{R}^m.$ Suppose there exists a fixed zero $[\bar{A}_{22}]_{ii}=0,$ then $v_i^T\tilde{A}_{22}v_i=[\tilde{A}_{22}]_{ii}=[\bar{A}_{22}]_{ii}=0$, for all numerical realizations of $\bar{A}_{22};$ hence, $\tilde{A}_{22}$ is not negative definite. 
\end{proof}

\begin{proof}[Proof of Theorem~\ref{con_SS}]
	%$(\Longleftarrow)$ For a set $\mathcal{S}\subseteq\mathcal{X}$, $\mathcal{N(S)}=\{i\in\mathcal{X\cup U}:(i,j)\in\mathcal{E},j\in\mathcal{S}\}$. 
\emph{(If)} Without loss of generality, suppose $(\bar{A},\bar{B})$ can be transformed to the form of \eqref{eq}. Suppose for $\forall \mathcal{S}\subseteq \mathcal{X}_r$, $\left|\mathcal{N(S)}\right|\ge \left|\mathcal{S}\right|$, then the input reachable subsystem $(\bar{A}_{11},\bar{B}_1)$ is structurally controllable. If for $\forall x_i \in\mathcal{X}_u$, $x_i$ has self-loop in $\mathcal{D}(\bar{A},\bar{B})$, then $[\bar{A}]_{ii}$ is a $\star$-entry. Let us assign negative numerical weights to all the $\star$-entries of $\bar{A}$ that correspond to the self-loop of all $x_i\in\mathcal{X}_u$. Then, the input-unreachable part of the system, $\tilde{A}_{22}$, is a negative definite diagonal matrix. Thus, we have shown that there exists a numerical realization $(\tilde{A},\tilde{B})$, such that the uncontrollable part is asymptotically stable. Hence, the system is structurally stabilizable.
	
	\emph{(Only if)} The necessity can be proved by contrapositive. Suppose there exists a state vertex $x_i\in\mathcal{X}_u$ that $[\bar{A}]_{ii}=0$, then, by Lemma~\ref{lem_1} any numerical realization $(\tilde{A},\tilde{B})$ has an uncontrollable non-negative eigenvalue. Furthermore, assume there exists $\mathcal{S}\subseteq\mathcal{X}$ such that $\left|\mathcal{N(S)}\right|<\left|\mathcal{S}\right|$, then by Lemma~\ref{lem_2}, $(\bar{A},\bar{B})$ is not structurally stabilizable.
\end{proof}

%\subsection{Proofs of Theorem~\ref{dim} and related results}
%To prove Lemma~\ref{theo6}, we may need to use the results on structural target controllability. Due to the space limitation, please refer to \cite{li2018tc} for more details.\vspace{-0.2cm}
%Before we derive the proof of Lemma~\ref{theo6} and Theorem~\ref{dim}, we first present a proposition and the proof of Proposition~\ref{theo4} can be found in \cite{li2018tc}.
%\begin{prop}\label{theo4}
%	Consider a structural pair $(\bar{A},\bar{B})$, with $\bar{A}$ being symmetrically structured, and a target set $\mathcal{T}\subseteq[n]$. Let $\mathcal{X_T}$ be the set of state vertices corresponding to $\mathcal{T}$ in $\mathcal{D}(\bar{A},\bar{B}).$ The structural pair $(\bar{A}, \bar{B})$ is structurally target controllable with respect to $\mathcal{T}$, if and only if, the following conditions hold simultaneously in $\mathcal{D}(\bar{A},\bar{B}):$\vspace{-0.05cm}
%	\begin{enumerate}
%		\item all the states vertices in $\mathcal{X_T}$ are input-reachable;
%		\item $\left|\mathcal{N(S)}\right|\ge\left|\mathcal{S}\right|$, $\forall \mathcal{S}\subseteq \mathcal{X_T}$.\vspace{-0.15cm}
%	\end{enumerate}
%\end{prop}

\begin{proof}[Proof of Lemma~\ref{theo6}]
%We prove the equality between generic dimension of controllable subspace of $(\bar{A},\bar{B})$ and the term-rank of $[\bar{A},\bar{B}]$ by 
Suppose $\textrm{t--rank}([\bar{A},\bar{B}])=k$, then there exists a set $\mathcal{T}\in[n]$, such that for $\forall \mathcal{S}\subseteq \mathcal{X_T}=\{x_i\in\mathcal{X}\colon i\in\mathcal{T}\}$, $\left|\mathcal{N(S)}\right|\ge\left|\mathcal{S}\right|$. By Theorem 2 in \cite{li2018tc}, $(\bar{A},\bar{B})$ is structurally target controllable\footnote{To prove Lemma~\ref{theo6}, here we use the results on structural target controllability. Due to page limitations, please refer to \cite{li2018tc} for more details.} with respect to $\mathcal{T}$, which implies that there exists a numerical realization $(\tilde{A},\tilde{B})$ with $[\mathbf{p}_{\tilde{A}},\mathbf{p}_{\tilde{B}}]\in V^c\cap W^c$, where $V$ and $W$ are proper varieties in $\mathbb{R}^{n_{\bar{A}}+n_{\bar{B}}}$ defined in the proof of Theorem 2 in \cite{li2018tc}, such that the dimension of the controllable subspace is $k$, i.e., almost surely the dimension of controllable subspace of a numerical realization $(\tilde{A},\tilde{B})$ is $k$. We have the generic dimension of controllable subspace of $(\bar{A},\bar{B})$, $d_c=\textrm{t--rank}([\bar{A},\bar{B}])$.
%On the other hand, suppose the generic dimension of controllable subspace $d_c>\textrm{t--rank}([\bar{A},\bar{B}])$, then there exists a target set $\mathcal{T}$ with $|\mathcal{T}|=d_c>\textrm{t--rank}([\bar{A},\bar{B}])$. By Theorem 5 in \cite{li2018tc}, it contradicts that 
%
%Therefore, $d_c=\textrm{t--rank}([\bar{A},\bar{B}])$.
\end{proof}

\begin{proof}[Proof of Theorem~\ref{dim}]
	Without loss of generality, there exists only two cases: either $(\bar{A},\bar{B})$ is irreducible or not. In the first case, by Lemma~\ref{lem:proper}, Lemma~\ref{lem_2} and Lemma~\ref{theo6}, the generic dimension of controllable subspace of $(\bar{A},\bar{B})$ is $\textrm{t--rank}([\bar{A},\bar{B}])$, and if $\textrm{t--rank}([\bar{A},\bar{B}])<n$, then for any numerical realization $(\tilde{A},\tilde{B})$, there are $(n-k)$ zero uncontrollable eigenvalues. Therefore, the maximum dimension of stabilizable subspace of $(\bar{A},\bar{B})$ equals to $\textrm{t--rank}([\bar{A},\bar{B}])$; In the other case, permute $(\bar{A},\bar{B})$ to the form of \eqref{eq} and let $k$ be the number of $\star$-entries in the diagonal of $\bar{A}_{22}$. Therefore, we have $\textrm{m-dim}(\bar{A},\bar{B})=\textrm{t-rank}(\bar{A}_{11},\bar{B}_1)+\textrm{m-dim}(\bar{A}_{22})$. 
	
	In what follows, we first prove that $\textrm{m-dim}(\bar{A}_{22})\le (n-p)-|\mathcal{S}|$, where $\mathcal{S}$ is a maximal independence set in $\mathcal{D}(\bar{A}_{22})$ and $\bar{A}_{22}\in\{0,\star\}^{(n-p)\times(n-p)}$. We then let $\{\mathcal{C}_i\}_{i=1}^{k}$ be a set of disjoint cycles in $\mathcal{D}(\bar{A}_{22})$ and we prove $\textrm{m-dim}(\bar{A}_{22})\ge \dfrac{1}{2}\sum_{i\in\mathcal{S}_e}|\mathcal{C}_i|+\dfrac{1}{2}\sum_{i\in\mathcal{S}_o}\left(|\mathcal{C}_i|+1\right)$, where $\mathcal{S}_e$ is the set of indexes of cycles with even length in $\{\mathcal{C}_i\}_{i=1}^k$and $\mathcal{S}_o$ is the set of indexes of cycles with odd length in $\{\mathcal{C}_i\}_{i=1}^k$.
	
	On one hand, suppose there exists a maximal independence set $\mathcal{S}$ in $\mathcal{D}(\bar{A}_{22})$, then we construct a new matrix $\mathcal{C}\in\{0,\star\}^{|\mathcal{S}|\times |\mathcal{S}|}$, which is composed by the rows and columns of $\bar{A}_{22}$ indexed by $\mathcal{S}$. We claimed that all the entries of $C$ are zeros, otherwise $\mathcal{S}$ is not an independence set. Thus, all the eigenvalues of $C$ are zeros. By Eigenvalue Interlacing Theorem~\cite[Page 246]{horn1990matrix}, we conclude that for any numerical realization $\tilde{A}_{22}$ there are at least $|\mathcal{S}|$ non-negative eigenvalues, i.e., there exists at most $(n-p)-|\mathcal{S}|$ strictly negative eigenvalues for any numerical realization $\tilde{A}_{22}$. Hence, $\textrm{m-dim}(\bar{A}_{22})\le (n-p)+|\mathcal{S}|$.
	
	On the other hand, suppose there exists a set of vertex-disjoint cycles $\{\mathcal{C}_i\}_{i=1}^{k}$ in $\mathcal{D}(\bar{A}_{22})$. Let us denote by $\mathcal{C}_i$ the $i$-th cycle in $\{\mathcal{C}_1,\dots,\mathcal{C}_k\}$. Moreover, without loss of generality, we could permute $\{\mathcal{C}_i\}_{i=1}^k$ such that the first $|\mathcal{S}_e|$ are the cycles with even lengths and the rest are cycles with odd lengths. %a set of $k$ vertices in $G(\bar{A})$. 
	{Note that by definition, there is a one-to-one correspondence between the edge in $\mathcal{D}(\bar{A}_{22})$ and the $\star$-entry in $\bar{A}_{22}$. From this observation, we denote by $\bar{A}_i\in\{0,\star\}^{|\mathcal{C}_i|\times |\mathcal{C}_i|}$ the square submatrix formed by collecting rows and columns in $\bar{A}_{22}$ corresponding to the indexes of vertices in the cycle $\mathcal{C}_i$. }	 	
	We let all the $\star$-entries of $\bar{A}_{22}$ be zero, except for $\star$-entries corresponding to edges in $\{\mathcal{C}_i\}_{i=1}^k$. Hence, there exists a permutation matrix $P$ and numerical realization $\tilde{A}_{22}$, such that $P\tilde{A}_{22}P^{-1}$ is a block diagonal matrix,%\vspace{-0.1cm}
	\begin{equation}\label{decomposition}\small
	P\tilde{A}_{22}P^{-1}=\begin{bmatrix}
	\tilde{A}_{1}&\mathbf{0}&\cdots&\mathbf{0}\\%\vspace{-0.19cm}
	\mathbf{0}&\tilde{A}_2&\cdots&\mathbf{0}\\
	\vdots & \vdots & \ddots&\vdots\\
	\mathbf{0} & \mathbf{0} &\cdots&\tilde{A}_k
	\end{bmatrix}.
	\end{equation}%\vspace{-0.3cm}
	
	Following the same technique in the proof of Lemma 2 in \cite{li2018tc}, for the submatrix $\bar{A}_i$ corresponding to the even-length cycle, i.e., $|\mathcal{C}_i|=2\ell$, we can construct a numerical realization $\tilde{A}_{i}$ with $\ell$ strictly negative eigenvalues. For the submatrix $\bar{A}_{i}$ corresponding to the odd-length cycle, i.e., $|\mathcal{C}_i|=2\ell-1$, we can construct a numerical realization with $\ell$ strictly negative eigenvalues. Hence, we can construct a numerical realization $\tilde{A}_{22}$ with $\dfrac{1}{2}\sum_{i\in\mathcal{S}_e}|\mathcal{C}_i|+\dfrac{1}{2}\sum_{i\in\mathcal{S}_o}\left(|\mathcal{C}_i|+1\right)$ strictly negative eigenvalues, i.e., $\textrm{m-dim}(\bar{A}_{22})\ge \dfrac{1}{2}\sum_{i\in\mathcal{S}_e}|\mathcal{C}_i|+\dfrac{1}{2}\sum_{i\in\mathcal{S}_o}\left(|\mathcal{C}_i|+1\right)$.
\end{proof}

%\subsection{Proofs of Theorem~\ref{nphard} and Theorem~\ref{thm:partial reduction}}
	\begin{figure}[t!]
	\centering
	\includegraphics[width=0.25\textwidth]{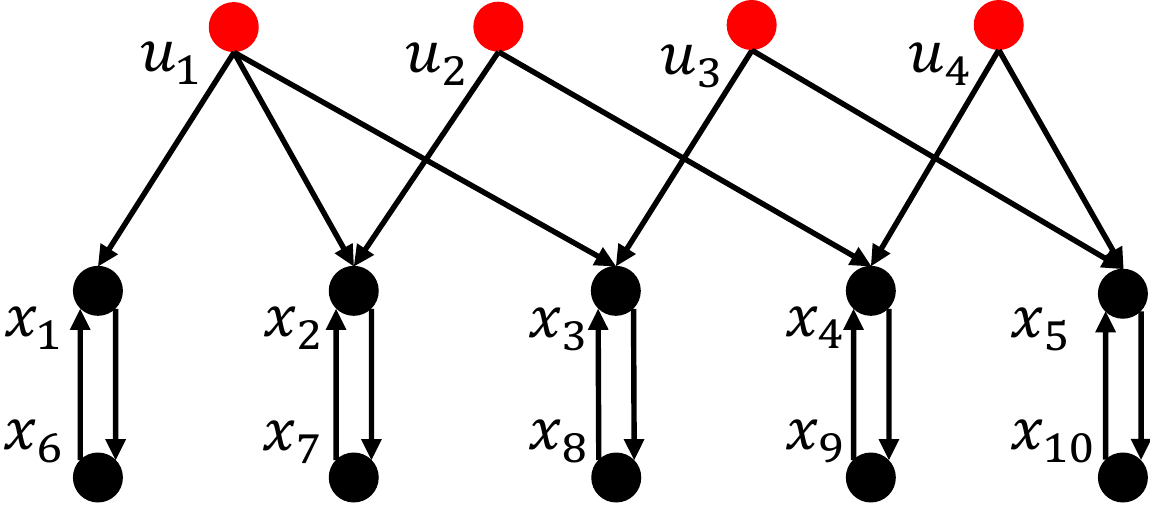}
	\caption{Example of the construction of $\mathcal{D}(\mathcal{X\cup U},\mathcal{E_{X,X}}\cup \mathcal{E_{U,X}})$ in the proof of Theorem~\ref{nphard}. Suppose we have a finite universe set $\mathcal{U_S}=\bigcup_{\ell=1}^4 \mathcal{S}_{\ell}$, where $\mathcal{U_S}=\{1,2,3,4,5\},\mathcal{S}_1=\{1,2,3\},\mathcal{S}_2=\{2,4\},\mathcal{S}_3=\{3,5\},\mathcal{S}_4=\{4,5\}$. From the given set $\mathcal{U_S}=\bigcup_{\ell=1}^4 \mathcal{S}_{\ell}$. We construct the state vertex set $\mathcal{X}=\{x_i\}_{i=1}^{10}$, and the input vertex set $\mathcal{U}=\{u_i\}_{i=1}^{4}$. The black and red vertices in Figure~\ref{fig1} are the state and input vertices in $\mathcal{D}(\mathcal{X\cup U},\mathcal{E_{X,X}}\cup \mathcal{E_{U,X}})$, respectively.}
	\label{fig1}
\end{figure}
\begin{proof}[Proof of Theorem~\ref{nphard}]
	\iffalse
	First, by Theorem~\ref{dim}, given a structural pair $(\bar{A},\bar{B})$, it takes polynomial time to compute the maximum dimension of the stabilizable subspace, which implies that Problem~\ref{p2} is in NP. 
	\fi
	We prove the NP-hardness of Problem~\ref{p2} by (polynomially) reducing Min-k-Union problem to instances of Problem~\ref{p2}.
	
	Suppose that we have a universe set $\mathcal{U_S}=\{\mathcal{S}_{\ell}\}_{\ell=1}^p$, and an integer $k\in\mathbb{Z}^+$, for which we need to select $k$ subsets in $\{\mathcal{S}_{\ell}\}_{\ell=1}^p$ such that $|\bigcup_{i=1}^k\mathcal{S}_{\ell_i}|$ is minimized. Let $n=|\mathcal{U_S}|$ and define the state vertex set as $\mathcal{X}=\{x_i\}_{i=1}^{2n}$, and input vertex set as $\mathcal{U}=\{u_i\}_{i=1}^{p}$. Next, we can construct a set of directed edges between state vertices, $\mathcal{E_{X,X}}=\{(x_i,x_{i+n}),(x_{i+n},x_i)\}_{i=1}^n$, and a set of directed edges between input and state vertices, $\mathcal{E_{U,X}}=\{(u_i,x_j)\colon i\in[p],j\in\mathcal{S}_i\}$ -- see Figure~\ref{fig1} as an example for such a construction. In the constructed graph $\mathcal{D}(\mathcal{X\cup U, E_{X,X}}\cup \mathcal{E_{U,X}})$, we have $|\mathcal{N(S)}|\ge |\mathcal{S}|, \forall \mathcal{S\subseteq X}$, and all $x_i\in\mathcal{X}$ are reachable. 
	
	{\color{black}From the graph $\mathcal{D(X\cup U,E_{X,X}\cup E_{U,X})}$, we construct the symmetrically structured matrix $\bar{A}\in\{0,\star\}^{2n\times 2n}$ such that $[\bar{A}]_{ij}=\star$ if $\{x_j,x_i\}\in\mathcal{E_{X,X}}$ and $[\bar{A}]_{ij}=0$ otherwise. We also construct $\bar{B}\in\{0,\star\}^{2n\times p}$ such that $[\bar{B}]_{ij}=\star$ if $\{u_j,x_i\}\in\mathcal{E_{U,X}}$ and $[\bar{B}]_{ij}=0$ otherwise. We can verify that the maximum stabilizable subspace for the constructed symmetrically structured matrix $\bar{A}\in\{0,\star\}^{2n\times 2n}$ is $n$ and $\textrm{m-dim}(\bar{A},\bar{B}(\mathcal{J}))=n+\dfrac{1}{2}|\cup_{j\in\mathcal{J}}\mathcal{S}_j|$.} %Hence, the problem of optimally removing a fixed number of inputs such that the generic dimension of stabilizable subspace is minimized becomes to optimally remove a fixed number of inputs such that the total number of reachable state vertices is minimized. %the structural pair $(\bar{A},\bar{B})$ cooresponding to the graph $G(\mathcal{X\cup U,E_{X,X}\cup E_{U,X}})$ is structurally stabilizable. 
	Let the attack budget be $c=p-k$. In our constructed instance of Problem~\ref{p2}, we aim to remove $c$ actuators from $\{u_i\}_{i=1}^p$ such that the maximum dimension of the stabilizable subspace is minimized. Subsequently, we claim that an optimal solution of the constructed instance of Problem~\ref{p2} enables us to retrieve an optimal solution to the Min-k-Union problem.
	
	%Subsequently, we show that a feasible solution to Problem~\ref{p2} leads to a feasible solution of Min-k-Union problem, and a minimum solution of Problem~\ref{p2} yields an optimal solution of Min-k-Union problem. 
	Suppose we have a feasible solution $\mathcal{U}_r=\{u_{\ell_i}\}_{i=1}^{p-k}$. Then if we consider $\mathcal{L}=[p]\setminus\{\ell_i\}_{i=1}^{p-k}$, we have that $\mathcal{L}$ is a feasible solution of Min-k-Union problem. Moreover, suppose $\mathcal{U}^*_r=\{u_{\ell_i}\}_{i=1}^{p-k}$ is a minimum solution to Problem~\ref{p2}, but $\mathcal{L}=[p]\setminus\{\ell_i\}_{i=1}^{p-k}$ is not an optimal solution of Min-k-Union problem, then $\mathcal{L}'=\{\eta_i\}_{i=1}^{k}$ would be a solution to Min-k-Union problem such that $|\bigcup_{i=1}^k \mathcal{S}_{\eta_i}|<|\bigcup_{i\in\mathcal{L}}\mathcal{S}_{i}|$. Next, let $\mathcal{U}'_r=\{u_i\in\mathcal{U}\colon i\in[p]\setminus\mathcal{L}'\}$ and notice that the maximum stabilizable subspace by removing $\mathcal{U}'_r$ is smaller than the maximum stabilizable subspace when removing $\mathcal{U}^*_r$, which contradicts $\mathcal{U}^*_r$ is an optimal solution.
%	
%	Therefore, the Problem~\ref{p2} can be polynomially reduced to Min-k-Union problem. The Problem~\ref{p2} is NP-hard.
	% $n+\mu$, where $\mu=|\bigcup_{i=1}^k\mathcal{S}_{\ell_i}|$ and $\{\mathcal{S}_{\ell_i}\}_{i=1}^k$ is an optimal solution of the minimum $k$-union problem. %In our problem setting, there are $p$ actuators, and the problem of optimally removing $c$ actuators such that the total number of reachable state vertices is minimized becomes to % Suppose $c=p-k$, we declaim that the generic dimension of stabilizable subspsace corresponding to the optimal solution $\mathcal{U}_r^*$ equals to $\mu$, where $\mu$ is an optimal solution of the given minimum k-union problem, i.e., $\textrm{g--dim}(\bar{A},\bar{B}_{\mathcal{U\setminus U}_r})=\mu$.
	%	
	%	On one hand, suppose the $\{\mathcal{S}_{\ell_i}\}_{i=1}^{k}$ is an optimal solution of minimum $k$-union problem, and $|\bigcup_{i=1}^{k}\mathcal{S}_{\ell_i}|=\mu$ is the cardinality of cooresponding union of $k$ subsets. Then, we will have the set of unreachable state vertices with no self-loop is $\{x_{i}\in\mathcal{X}\colon i\in\mathcal{U_S}\setminus(\bigcup_{i=1}^k\mathcal{S}_{\ell_i})\}$. Hence, the generic dimension of stabilizable subspace is $n+\mu$.
	%	
	%	On the other hand, suppose an optimal solution of Problem~\ref{p2} yields the generic dimension of stabilizable subspace to be $(n+\mu)$. Then we have there are $n-\mu$ 
\end{proof}
\begin{proof}[Proof of Theorem~\ref{thm:partial reduction}]
	%(\emph{i}) 
{\color{black}	Consider an instance of Problem~\ref{p2} under Assumption~\ref{asm:perfect matching}. We associate the structural pair $(\bar{A},\bar{B})$, where $\bar{A}\in\{0,\star\}^{n\times n}$ and $\bar{B}\in\{0,\star\}^{n\times m}$, with a digraph $\mathcal{D}(\bar{A},\bar{B})$. Denote by $\{\mathcal{D}_{i}\}_{i=1}^p$ the set of vertices in the $i$th SCC in $\mathcal{D}(\bar{A},0)$. 
	
	Firstly, without loss of generality, we could assume that $\mathcal{D}_i=\{x_{d_{i-1}+1},x_{d_{i-1}+2},\cdots, x_{d_{i-1}+|\mathcal{D}_i|}\}$, with $d_0=0$ and $d_i=d_{i-1}+|\mathcal{D}_i|$. Secondly, for the $i$th SCC, we define the set $\mathcal{R}_i=\{r_{i-1}+1,r_{i-1}+2,\dots,r_{i-1}+(|\mathcal{D}_i|-\textrm{m-dim}(\bar{A}_{\mathcal{D}_i},0))\}$, with $r_0=0$ and $r_{i}=r_{i-1}+|\mathcal{D}_i|-\textrm{m-dim}(\bar{A}_{\mathcal{D}_i},0)$, where $\bar{A}_{\mathcal{D}_i}\in\{0,\star\}^{|\mathcal{D}_i|\times |\mathcal{D}_i|}$ is the symmetrically structured matrix corresponding to the $i$th SCC. Finally, for the $j$th control input, we construct the set $\mathcal{S}_j=\{\bigcup_{i\in \mathcal{I}}\mathcal{R}_i|\  \mathcal{I}\textrm{ is the set of SCCs reachable from }u_j\}$.}
	
	By the above construction and Assumption~\ref{asm:perfect matching}, we have
	\begin{equation}
		\textrm{m-dim}(\bar{A},\bar{B}(\mathcal{J}))=\textrm{m-dim}(\bar{A},0)+|\bigcup_{j\in\mathcal{J}}\mathcal{S}_j|.
	\end{equation}
	
	\iffalse
	For each $i\in[m]$, we let $\mathcal{S}_i$ be the set of state vertices which are reachable from the input $u_i$. By Assumption~\ref{asm:perfect matching}, we have
	\begin{equation}\footnotesize\label{eq:mm}
	\textrm{m-dim}(\bar{A},\bar{B}(\mathcal{J}))=|\bigcup_{j\in\mathcal{J}}\mathcal{S}_j|.
	\end{equation}
	\fi
We let $\mathcal{U_S}=\bigcup_{i=1}^m\mathcal{S}_i$. Suppose the budget in Problem~\ref{p2} is $k$, then Problem~\ref{p2} is to find $(m-k)$ sets $\mathcal{S}_{\ell_1},\cdots,\mathcal{S}_{\ell_{m-k}}$ among $\{\mathcal{S}_i\}_{i=1}^m$ such that $|\bigcup_{i=1}^{m-k}\mathcal{S}_{\ell_i}|$ %, i.e., the number of reachable state vertices, 
is minimized.	By Definition~\ref{df:mku}, we see that in this case Problem~\ref{p2} is equivalent to the Min-k-Union problem, in which we are given sets $\{\mathcal{S}_i\}_{i=1}^{m}$ and we aim to find $(m-k)$ sets $\{\mathcal{S}_{\ell_{i}}\}_{i=1}^{m-k}$, $\{\ell_i\}_{i=1}^{m-k}\subseteq\{1,2,\cdots,m\}$, such that $|\bigcup_{i=1}^{m-k}\mathcal{S}_{\ell_i}|$ is minimized. If there exists a $\rho(m)$-approximation algorithm for the Min-k-Union problem, i.e., $|\bigcup_{j\in\mathcal{J}}\mathcal{S}_j|\le \rho(m)|\bigcup_{j\in\mathcal{J}^*}\mathcal{S}_j|$, then,  %$\bigcup_{j\in\mathcal{J}}\mathcal{S}_j$ is the set of all input-reachable state vertices, and
	\begin{equation}\footnotesize\label{eq:m}
	\begin{aligned}
	\textrm{m-dim}&(\bar{A},\bar{B}(\mathcal{J}))=\textrm{m-dim}(\bar{A},0)+|\bigcup_{j\in\mathcal{J}}\mathcal{S}_j|\\
	&\le\textrm{m-dim}(\bar{A},0)+ \rho(m)\cdot( |\bigcup_{j\in\mathcal{J}^*}\mathcal{S}_j|)\\
	&\le\rho(m)\cdot\textrm{m-dim}(\bar{A},\bar{B}(\mathcal{J}^*)),
	\end{aligned}
	%			\textrm{m-dim}(\bar{A},\bar{B}(\mathcal{J}))&=|\mathcal{X}_0|+|\bigcup_{j\in\mathcal{J}}\mathcal{S}_j|\\
	%			f(m)\cdot \textrm{m-dim}(\bar{A},\bar{B}(\mathcal{J}))&\le|\mathcal{X}_0|+|\bigcup_{j\in\mathcal{J}}\mathcal{S}_j|,\ \forall \mathcal{J}\subseteq [m],
	\end{equation}
	where $\mathcal{J}^*$ is an optimal solution to the Min-k-Union problem. From the above reasoning, we have that $\bar{B}(\mathcal{J}^*)$ is also an optimal solution to Problem~\ref{p3} and $\textrm{m-dim}(\bar{A},\bar{B}(\mathcal{J}))\le\rho(m)\cdot\textrm{m-dim}(\bar{A},\bar{B}(\mathcal{J}^*))$.
	%	where $\mathcal{X}_0$ is the set of state vertices with self-loop in $\mathcal{D}(\bar{A},\bar{B})$. %, and $\{\mathcal{S}_j\}_{j=1}^m$ are the sets in the Min-k-Union problem. 
	\iffalse
	(\emph{ii}) Conversely, if there exists a $\rho(m)$-approximation algorithm for Problem~\ref{p2} under Assumption~\ref{asm:perfect matching}, then from \eqref{eq:m} and the above reasoning, there also exists a $\rho(m)$-approximation algorithm for Min-k-Union problem.
	\fi
\end{proof}
\begin{proof}[Sketch of Proof of Theorem~\ref{nphard2}]
	%Since a solution of Problem~\ref{p3} can be verified in polynomial time, the Problem~\ref{p3} is in NP. 
	We can prove the NP-hardness by reducing a general instance of the Max-k-Union problem to an instance of Problem~\ref{p3}. Suppose we have a ground set $\mathcal{U_S}=\{\mathcal{S}_{\ell}\}_{\ell=1}^p$, and an integer $k\in\mathbb{N}$. The constrained maximum set coverage problem is to select $k$ subsets in $\mathcal{U_S}$ such that $|\bigcup_{i=1}^k\mathcal{S}_{\ell_i}|$ is maximized. Following a similar construction and reasoning taken in the proof of Theorem~\ref{nphard}, we can prove that the Max-k-Union problem can be reduced to Problem~\ref{p3} in polynomial time.% The Problem~\ref{p3} is NP-hard.
\end{proof}

\begin{proof}[Proof of Theorem~\ref{them_alg}]
Consider a structural pair $(\bar{A},\bar{B})$, where $\bar{A}\in\{0,\star\}^{n\times n}$ is symmetrically structured and $\bar{B}\in\{0,\star\}^{n\times m}$ is structured. We let $\mathcal{U}$ denote the input vertices corresponding columns of $\bar{B}$, and let $\mathcal{U}_{can}$, where $|\mathcal{U}_{can}|=m'$, be the set of new actuators that can be added to the system. We associate with the set $\mathcal{U}_{can}$ the structured matrix $\bar{B}_{\mathcal{U}_{can}}\in\{0,\star\}^{n\times m'}$. Define a function $f\colon \mathcal{J}\subseteq[m']\to \textrm{m--dim}(\bar{A},[\bar{B},\bar{B}_{\mathcal{U}_{can}}(\mathcal{J})])$. We first prove that $f(\mathcal{J})$ is a submodular function, and then we show that Algorithm~\ref{greedy} returns a $(1-1/e)$ approximation solution.
	
{\color{black}	Before we proceed, we construct a few sets. We denote by $\{\mathcal{D}_i\}_{i=1}^p$ the set of vertices in the $i$th input-unreachable SCC in $\mathcal{D}(\bar{A},\bar{B})$. Without loss of generality, we could assume that $\mathcal{D}_i=\{x_{d_{i-1}+1},x_{d_{i-1}+2},\cdots, x_{d_{i-1}+|\mathcal{D}_i|}\}$, with $d_0=0$ and $d_i=d_{i-1}+|\mathcal{D}_i|$. Then, for the $i$th unreachable SCC, we define the set $\mathcal{R}_i=\{r_{i-1}+1,r_{i-1}+2,\dots,r_{i-1}+(\textrm{t-rank}(\bar{A}_{\mathcal{D}_i})-\textrm{m-dim}(\bar{A}_{\mathcal{D}_i},0))\}$, with $r_0=0$ and $r_{i}=r_{i-1}+\textrm{t-rank}(\bar{A}_{\mathcal{D}_i})-\textrm{m-dim}(\bar{A}_{\mathcal{D}_i},0)$, where $\bar{A}_{\mathcal{D}_i}\in\{0,\star\}^{|\mathcal{D}_i|\times |\ \mathcal{D}_i|}$ is the symmetrically structured matrix corresponding to the $i$th unreachable SCC. Finally, for the $j$th control input $u_j$ in $\mathcal{U}_{can}$, we construct the set $\mathcal{S}_j=\{\bigcup_{i\in \mathcal{I}}\mathcal{R}_i|\  \mathcal{I}\textrm{ is the set of unreachable SCCs in } \mathcal{D}(\bar{A},\bar{B})$ $\textrm{but reachable from }u_j\}$.

Let $q({\mathcal{J}})$ be the total number of state vertices which are right-unmatched in $\mathcal{B}(\bar{A},\bar{B})$ but right-matched in $\mathcal{B}(\bar{A},[\bar{B},\bar{B}_{\mathcal{U}_{can}}(\mathcal{J})])$. By definition, we have
\begin{equation}
	f(\mathcal{J})=\textrm{m-dim}(\bar{A},\bar{B})+|\bigcup_{j\in\mathcal{J}}\mathcal{S}_j|+q(\mathcal{J}).
\end{equation}
which implies $f(\mathcal{J})$ is a monotonically increasing function of $\mathcal{J}\subseteq[m']$.

Furthermore, consider two sets $\mathcal{J}_1$, $\mathcal{J}_2$, where $\mathcal{J}_1\subseteq\mathcal{J}_2\subseteq[m']$. Suppose $j\in[m']\setminus \mathcal{J}_2$ and denote by $\mathcal{J}_1'=\mathcal{J}_1\cup\{j\}$ and $\mathcal{J}_2'=\mathcal{J}_2\cup\{j\}$, then
\begin{equation}
	f(\mathcal{J}_1')-f(\mathcal{J}_1)=|\bigcup_{j\in\mathcal{J}_1'}\mathcal{S}_j|-|\bigcup_{j\in\mathcal{J}_1}\mathcal{S}_j|+q(\mathcal{J}_1')-q(\mathcal{J}_1),
\end{equation}
and 
\begin{equation}
	f(\mathcal{J}_2')-f(\mathcal{J}_2)=|\bigcup_{j\in\mathcal{J}_2'}\mathcal{S}_j|-|\bigcup_{j\in\mathcal{J}_2}\mathcal{S}_j|+q(\mathcal{J}_2')-q(\mathcal{J}_2).
\end{equation}

On one hand, suppose $q(\mathcal{J}_1')-q(\mathcal{J}_1)=1$, then $q(\mathcal{J}_2')-q(\mathcal{J}_2)=1$ or $0$; On the other hand, suppose $q(\mathcal{J}_1')-q(\mathcal{J}_1)=0$, then $q(\mathcal{J}_2')-q(\mathcal{J}_2)=0$. Recall that the set coverage function is a submodular function, i.e., $|\bigcup_{j\in\mathcal{J}_1'}\mathcal{S}_j|-|\bigcup_{j\in\mathcal{J}_1}\mathcal{S}_j|\ge|\bigcup_{j\in\mathcal{J}_2'}\mathcal{S}_j|-|\bigcup_{j\in\mathcal{J}_2}\mathcal{S}_j|$. Therefore, we have 
\begin{equation}
	f(\mathcal{J}_1')-f(\mathcal{J}_1)\ge f(\mathcal{J}_2')-f(\mathcal{J}_2),
\end{equation}
which implies that $f(\mathcal{J})$ is a monotonically increasing submodular function.
}

	Because $f(\mathcal{J})$ is a monotonically increasing submodular function, by a similar technique taken in the proof of \cite[Proposition 5.1]{feige1998threshold}, we can show that Algorithm~\ref{greedy} returns a $(1-1/e)$-approximation solution to Problem~\ref{p3}.
\end{proof}

\footnotesize\bibliographystyle{ieeetran}
\bibliography{acc2019}
\end{document}